\documentclass{amsart}
\usepackage{amsfonts}
\usepackage{amsmath}
\usepackage{amssymb}
\usepackage{enumerate}
\usepackage{framed}
\usepackage{graphicx}
\usepackage{lscape}
\usepackage{bm}
\usepackage{color}
\usepackage{multicol}
\usepackage{amsthm} 
\usepackage{amsthm} 
\newtheorem{theorem}{Theorem}[section]

\newtheorem{lemma}[theorem]{Lemma}
\newtheorem{proposition}[theorem]{Proposition}

\newtheorem{remark}[theorem]{Remark}
\newtheorem{assumption}{Assumption}

\usepackage{algorithm}
\usepackage{subfigure}
\newcommand{\R}{\mathbb{R}}

\newcommand{\N}{\mathbb{N}}

\numberwithin{equation}{section}
\usepackage{xspace}
\newcommand{\Leja}{{L\'{e}ja}\xspace}

\newcommand{\Pade}{{Pad\'e}\xspace}
\newcommand{\SETDLR}{{(\textbf{SETD0})}\xspace}
\newcommand{\SETD}{{(\textbf{SETD1})}\xspace}
\newcommand{\HH}{\mathbb{H}}
\newcommand{\Dt}{\Delta t}
\newcommand{\thmref}[1]{{Theorem~\ref{#1}}}
\newcommand{\lemref}[1]{{Lemma~\ref{#1}}}

\newcommand{\figref}[1]{{Figure~\ref{#1}}}

\newcommand{\secref}[1]{{Section~\ref{#1}}}
\newcommand{\assref}[1]{{Assumption~\ref{#1}}}
\newcommand{\propref}[1]{{Proposition~\ref{#1}}}

\begin{document}
\title[Stochastic Exponential Integrators]{Stochastic Exponential Integrators for a Finite Element Discretization of  SPDEs}
\author{Gabriel J. Lord}
\email{G.J.Lord@hw.ac.uk}
\address{Department of Mathematics and the Maxwell Institute for Mathematical Sciences, Heriot Watt University, Edinburgh EH14 4AS, U.K }
\author{Antoine Tambue}
\email{at150@hw.ac.uk}
\address{Department of Mathematics and the Maxwell Institute for Mathematical Sciences, Heriot Watt University, Edinburgh EH14 4AS, U.K }

\begin{abstract}
 We consider the numerical approximation of general semilinear
 parabolic stochastic partial differential equations (SPDEs) driven by
 additive space-time noise. 
 In contrast to the standard  time stepping methods which 
 uses basic increments of the noise and the approximation of the
 exponential function  by a rational fraction, we introduce a new
 scheme, designed for finite elements, finite volumes or finite
 differences  space discretization,
 similar to the schemes in \cite{Jentzen3,Jentzen4} for spectral
 methods and \cite{GTambue} for finite element methods. 
 We use the projection operator, the smoothing
 effect of the positive definite self-adjoint operator  and 
 linear functionals of the noise in Fourier space to obtain  higher
 order approximations. 
 We consider  noise that is white in time  and either in $H^1$ or $H^2$
 in space and give convergence proofs in the mean square $L^{2}$ norm for
 a diffusion reaction equation and in mean  square $ H^{1}$ norm in
 the presence of an advection term. 
 For the exponential integrator we rely on computing the exponential of
 a non-diagonal matrix. In our numerical results we use two different
 efficient techniques: the real fast \Leja points and Krylov subspace
 techniques.  We present results for a linear  reaction diffusion equation  
 in two dimensions as well as a nonlinear example of two-dimensional
 stochastic advection diffusion reaction equation motivated from
 realistic porous media flow. 
\end{abstract}

\maketitle
\section{Introduction}
We consider the strong numerical approximation of 
Ito stochastic partial differential equations 
\begin{eqnarray}
  \label{adr}
  dX=(AX +F(X, \nabla X))dt + d W
\end{eqnarray}
in a Hilbert space $H= L^{2}(\Omega)$, where $\Omega\subset 
\mathbb{R}^{d}$ and $ t \in [0, T]$, $T>0$ and initial data
$X(0)=X_{0}$ is given. The linear operator $A$ is
the generator of an analytic semigroup $S(t):=e^{t A}, t\geq 0\;$ with eigenfunctions $e_i,\;i\in
\mathbb{N}^{d}$ and
$F$ is a nonlinear function. 
The noise can be
represented as a series in the eigenfunctions of the covariance
operator $Q$ and we assume for convenience  that $Q$ and $A$ have the same
eigenfunctions (without loss the generality as the orthogonal
projection can be used). In which case \cite{DaPZ}  we have 
\begin{eqnarray}
  \label{eq:W}
  W(x,t)=\underset{i \in
    \mathbb{N}^{d}}{\sum}\sqrt{q_{i}}e_{i}(x)\beta_{i}(t), 
\end{eqnarray}
where $q_i$, $i\in \mathbb{N}^{d}$ are the eigenvalues
of the covariance operator $Q$. The $\beta_{i}$ are
independent and identically distributed standard Brownian motions.

The study of numerical solutions of SPDEs is an active 
research area and there is an extensive literature on numerical
methods for SPDEs. Recent work by Jentzen and co-workers
\cite{Jentzen1,Jentzen2,Jentzen3,Jentzen4} uses the Taylor 
expansion and linear functionals of the noise for Fourier--Galerkin
discretizations of \eqref{adr}. In these schemes the diagonalization
of the operator $A$ through the discretization plays a key role.
Using a linear functional of the noise overcomes the order barrier
encountered using a standard increment of Wiener process \cite{Jentzen3}.
In \cite{GTambue} the use of linear functionals of the noise is
extended to finite--element discretizations (where the operator does
not diagonalize) with a semi-implicit Euler--Maruyama method.
In contrast to \cite{GTambue} here we consider two exponential based
methods for time-stepping as in \cite{LR,LT,Jentzen3,Jentzen4,KLNS}. 
We prove a strong convergence result for two versions of the scheme
with noise that is white in time and in $H^1$ and $H^2$ in
space that shows that the exponential integrators are more accurate
that the semi-implicit Euler-Maruyama method. 
Furthermore we have weaker restrictions on the regularity of the
initial data and high accuracy for linear problems comparing to the scheme
in  \cite{GTambue}.  The cost of the extra accuracy though is that to
implement these methods we need to compute the exponential of a
non--diagonal matrix, which is a notorious problem in numerical
analysis \cite{CMCVL}. 
However, new developments for both \Leja points and Krylov subspace
techniques \cite{kry,SID,Antoine,LE2,LE1,LE} have led to efficient
methods for computing matrix  exponentials.
Compared to the Fourier-Galerkin methods of
\cite{Jentzen1,Jentzen2,Jentzen3,Jentzen4} we gain the  flexibility of
finite element (or finite volume methods) to deal with complex boundary
conditions and we can apply well developed techniques such as
upwinding to deal with advection. 

We consider two examples of \eqref{adr} where $A$ is the second order
operator $D \varDelta $ and $D>0$ is the diffusion coefficient. For the
first example 
\begin{eqnarray}
  dX=\left(D \varDelta X - \lambda X \right)dt+ dW
\end{eqnarray}
where $\lambda$ is a constant, we can construct an exact solution. Our
second example is a stochastic advection diffusion reaction equation in
a heterogeneous porous media  
\begin{eqnarray}
  dX=\left(D \varDelta X -  \nabla \cdot (\textbf{q}
    X)-\frac{X}{X+1}\right)dt+ dW 
\end{eqnarray}
where $\textbf{q}$ is the Darcy velocity field \cite{sebastianb}.
In the linear example we take Neumann boundary conditions and for the
example from porous media we take a mixed Neumann--Dirichlet boundary
condition. 

The paper is organised as follows. In \secref{sec:scheme} we
present the two  numerical schemes based on the exponential
integrators and our assumptions on \eqref{adr}. We present and comment
on our convergence results. In \secref{sec:proofs}  we present
the proofs of our convergence theorems. 
We conclude in \secref{sec:sim} by presenting some simulations and
discuss implementation of these methods. 
\section{Numerical scheme and main results}
\label{sec:scheme}

We start by introducing our notation. We denote by $\Vert \cdot \Vert$
the norm associated to the standard inner product $(\cdot ,\cdot )$ of the
Hilbert space $H=L^{2}(\Omega)$ and $\Vert \cdot \Vert_{H^{m}(\Omega)}
$ the norm of the Sobolev space $H^{m}(\Omega)$, for $m\in\R$. 
For a Banach space $\mathcal{V}$ we denote by
$L(\mathcal{V})$  the set of bounded linear mapping  from
$\mathcal{V}$ to $\mathcal{V}$ and $L^{(2)}(\mathcal{V})$ the set of
bounded  bilinear mapping from  $\mathcal{V} \times \mathcal{V}$ to
$\mathbb{C}$.  We introduce further spaces and notation below as required. 

Consider the stochastic partial differential equation (\ref{adr}), 
under some technical assumptions it is well known (see
\cite{DaPZ,PrvtRcknr,Jentzen1} and references therein) that the
unique mild solution is given by  
\begin{eqnarray}
\label{eq1}
  X(t) &=& S(t)X_{0}+\int_{0}^{t}S(t-s)F(X(s))ds+O(t),
\end{eqnarray}
with the stochastic process $O$ given by the stochastic convolution 
\begin{eqnarray}
O(t)=\int_{0}^{t}S(t-s)dW(s).
\end{eqnarray}

We consider discretization of the spatial domain by a finite element 
triangulation.
Let $\mathcal{T}_{h}$ be a set of disjoint intervals  of $\Omega$ 
(for $d=1$), a triangulation of $\Omega$ (for $d=2$) or a set of
tetrahedra (for $d=3$). Let $V_{h}$ denote the space of continuous
functions that are piecewise linear over 
$\mathcal{T}_{h}$. To discretize in space we introduce two
projections. Our first projection operator $P_h$ is the
$L^{2}(\Omega)$ projection onto $V_{h}$ defined  for  $u \in L^{2}(\Omega)$ by 
\begin{eqnarray}
  (P_{h}u,\chi)=(u,\chi)\qquad \forall\;\chi \in V_{h}.
\end{eqnarray}
We can then define the operator $A_{h}: V_{h}\rightarrow V_{h}$, the discrete
analogue of $A$, by 
\begin{eqnarray}
( A_{h}\varphi,\chi)=(A\varphi,\chi)\qquad \varphi,\chi \in V_{h}.
\end{eqnarray}
We denote by $S_h(t):= e^{tA_h}$ the semigroup generated by the
operator $A_h$.
The second projection $P_{N}$, $N \in \mathbb{N} $ is the projection
onto a finite number of spectral modes ${e_i}$ defined for $u\in
L^{2}(\Omega)$ by  
$$u^N := P_{N}u=\sum_{i=1}^{N} (e_{i},u) e_{i}.$$
Furthermore we can project the operator $A$
$$A_N = P_N A \quad \text{and} \quad  S_{N}(t):=e^{t A_{N}}.$$

We discretize in space using finite elements and project the noise
first onto a finite number of modes and then onto the finite element
space. The semi-discretized version of \eqref{adr} is to find the
process $X^{h}(t)=X^{h}(.,t) \in V_{h}$ such  that 
\begin{eqnarray}
\label{dadr}
  dX^{h}=(A_{h}X^{h} +P_{h}F(X^{h}))dt + P_{h}P_{N}d W,
\quad  X^{h}(0)=P_{h}X_{0}.
\end{eqnarray}
The  mild solution of (\ref{dadr})  at  time $t_{m}=m \Delta t $,
$\Delta t>0$ is given by  
\begin{eqnarray*}
\label{dmild}
  X^{h}(t_{m})=S_{h}(t_{m})P_{h}X_{0}+\int_{0}^{t_{m}} S_{h}(t_{m}-s)
  P_{h}F(X^{h}(s)) ds + \int_{0}^{t_{m}} S_{h}(t_{m}-s)P_{h}d
  W^{N}(s). 
\end{eqnarray*}
Given the mild solution at the time $t_{m}$, we can construct
the corresponding solution at $t_{m+1}$ as 
\begin{eqnarray}
\label{Xhint}
 X^{h}(t_{m+1}) &=& S_{h}(\Delta t)X^{h}(t_{m})+\int_{0}^{ \Delta t}
 S_{h}( \Delta t-s) P_{h}F(X^{h}(s+t_{m}))ds \nonumber \\
& & + \int_{t_{m}}^{t_{m+1}} S_{h}(t_{m+1}-s)P_{h}d W^{N}(s). 
\end{eqnarray}

For our first numerical scheme \SETD, we use the following
approximations 
\begin{eqnarray*}
  F(X^{h}( t_{m}+s)) &\approx&  F(X^{h}( t_{m}))\qquad s \in [0,\;
  \Delta t],
\end{eqnarray*}
and 
\begin{eqnarray*}
  \int_{t_{m}}^{t_{m+1}} S_{h}(t_{m+1}-s)P_{h}d W^{N}(s)&\approx&
  P_{h}\int_{t_{m}}^{t_{m+1}} S_{N}(t_{m+1}-s)d W^{N}(s)\\
  &=&P_{h} P_{N}
  \int_{t_{m}}^{t_{m+1}}S(t_{m+1}-s)dW(s).
\end{eqnarray*}
Then we approximate $X_{m}^{h}$ of  $X(m \Delta t)$ by 
\begin{eqnarray}
\label{new}
 X_{m+1}^{h}&=&e^{\Delta t A_{h}}X_{m}^{h}+\Delta t\varphi_{1}(\Delta t A_{h})P_{h}F(X_{m}^{h}) \\ &+& P_{h}\int_{t_{m}}^{t_{m+1}} 
 e^{(t_{m+1}-s)A_{N}}d W^{N}(s)\nonumber
\end{eqnarray}
where 
$$\varphi_{1}(\Delta t A_{h})=(\Delta t\, A_{h})^{-1}\left( e^{\Delta
    t A_{h}}-I\right)= \frac{1}{\Delta t}\int_{0}^{\Delta t}
e^{(\Delta t- s)A_{h}}ds. $$ 
For efficiency to avoid computing two matrix exponentials we can
rewrite the scheme (\ref{new}) as 
$$
  X_{m+1}^{h} =  X_{m}^{h}+\Delta t \varphi_{1}(\Delta t
 A_{h})\left(A_{h}X_{m}^{h}\\+P_{h}F(X_{m}^{h})\right)+P_{h}\int_{t_{m}}^{t_{m+1}} e^{(t_{m+1}-s)A_{N}}d W^{N}(s).
$$
We call this scheme (\textbf{SETD1}).

Our second numerical method \SETDLR is similar to the one in
\cite{LR,LT,KLNS}.  It is based on approximating the deterministic
integral in \eqref{Xhint} at the left--hand endpoint of each partition
and the stochastic integral as 
follows
\begin{eqnarray*}
\int_{t_{m}}^{t_{m+1}} S_{h}(t_{m+1}-s)P_{h}d W^{N}(s) & \approx &
P_{h}\int_{t_{m}}^{t_{m+1}} S_{N}(t_{m+1}-s)d W^{N}(s)\\&=&P_{h} P_{N}
\int_{t_{m}}^{t_{m+1}}S(t_{m+1}-s)dW(s). 
\end{eqnarray*}
With this we can define the \SETDLR approximation $Y_{m}^{h}$ of  $X(m
\Delta t)$ by  
\begin{eqnarray}
\label{new0}
\qquad Y_{m+1}^{h}=\varphi_{0}(\Delta t
A_{h})\left(Y_{m}^{h}+\Dt P_{h} F(Y_{m}^{h})\right)+
P_{h}\int_{t_{m}}^{t_{m+1}} e^{(t_{m+1}-s)A_{N}}d W^{N}(s) 
\end{eqnarray}
where   
$$\varphi_{0}(\Delta t A_{h})=e^{\Delta t A_{h}}.$$

If we project  the eigenfunctions of $Q$ onto  the eigenfunctions of
the linear operator $A$ then by a Fourier spectral method the process 
$$
 \widehat{O}_{k}=\int_{t_{k}}^{t_{k+1}} e^{(t_{k+1}-s)A_{N}}d W^{N}(s)
$$
is reduced to an Ornstein--Uhlenbeck process in each Fourier mode as
in \cite{Jentzen3} and we therefore know the exact variance in each
mode. We comment further on the implementation in \secref{sec:sim}.
%
%
We describe now in detail the assumptions that we make
on the linear operator $A$, on our finite element discretization, the
nonlinear term $F$ and the noise $dW$.
\begin{assumption}
  \label{assumption1} 
  Let the linear operator $-A$ be a self adjoint positive
  definite operator and $A$ generate an analytic semigroup $S$. Then
  there exist sequences of  positive 
  real eigenvalues $\{\lambda_{n}\}_{n\in \mathbb{N}^{d}}$
  and an orthonormal basis in $H$ of eigenfunctions 
  $\{e_{i}\}_{i \in\mathbb{N}^{d} }$ such that the linear operator
  $-A:\mathcal{D}(-A)\subset H \rightarrow H $ is represented as 
  \begin{eqnarray*}
    -Av=\underset{i \in \mathbb{N}^{d}}{\sum}\lambda_{i}(
    e_{i},v)e_{i}\qquad \forall \quad v \in  \mathcal{D}(-A)  
  \end{eqnarray*}
  where the domain of $-A$, $\mathcal{D}(-A)=\{ v \in H :\underset{
    i\in \mathbb{N}^{d}}{\sum}\lambda_{i}^{2} |( e_{i},v) |<
  \infty\}$ and $ \underset{i \in \mathbb{N}^{d}}{\inf}\lambda_{i} > 0$ .
\end{assumption}
Note that for convenience of presentation we take $A$ to be a second order
operator as this simplifies notation for the norm equivalence below. 
Similar result hold, however, for higher order operators. 
We recall some basic properties of the semigroup $S(t)$ generated by
$A$ that may be found for example in \cite{Henry,Pazy}. 

\begin{proposition}[\cite{Henry}]
\label{prop1}
 Let $\beta \geq 0 $ and $0\leq \gamma \leq 1$, then  there exist
 $C>0$ such that 
\begin{eqnarray*}
 \Vert (-A)^{\beta}S(t)\Vert &\leq& C t^{-\beta}\;\;\;\;\; \text {for }\;\;\; t>0\\
  \Vert (-A)^{-\gamma}( \text{I}-S(t))\Vert &\leq& C t^{\gamma} \;\;\;\;\; \text {for }\;\;\; t\geq0.
\end{eqnarray*}
In addition,
\begin{eqnarray*}
(-A)^{\beta}S(t)&=& S(t)(-A)^{\beta}\quad \text{on}\quad \mathcal{D}((-A)^{\beta} )\\
\text{If}\;\;\; \beta &\geq& \gamma \quad \text{then}\quad
\mathcal{D}((-A)^{\beta} )\subset \mathcal{D}((-A)^{\gamma} ). 
\end{eqnarray*}
\end{proposition}

We introduce two spaces $\HH$ and $V$ where $\HH\subset V $ that depend
on the choice of the boundary conditions.
For Dirichlet boundary conditions we let 
\begin{eqnarray*}
V= \HH= H_{0}^{1}(\Omega)=\{v\in H^{1}(\Omega): v=0\;\;
\text{on}\;\;\partial \Omega\}, 
\end{eqnarray*}
and for Robin boundary conditions, for which Neumann boundary conditions
are a particular case, $V=  H^{1}(\Omega)$ and
\begin{eqnarray*}
\HH = \left\lbrace v\in H^{1}(\Omega): \partial v/\partial
  \nu_{A}+\sigma v=0\quad \text{on}\quad \partial \Omega\right\rbrace, 
 \end{eqnarray*}
see \cite{lions} for details. Functions in $\HH$ satisfy the boundary
conditions and with $\HH$ in hand we can characterize the
domain of the operator $(-A)^{r/2}$ and have the following norm
equivalence \cite{Stig,ElliottLarsson} for $r=1,2$
\begin{eqnarray*}
\Vert v \Vert_{H^{r}(\Omega)} \equiv \Vert (-A)^{r/2} v
\Vert_{L^{2}(\Omega)}:=\Vert  v \Vert_{r}, \qquad 
\forall v\in \mathcal{D}((-A)^{r/2})=  \HH\cap H^{r}(\Omega).
\end{eqnarray*} 
We now introduce our assumptions on the spatial domain and finite
element space $V_h$. We consider the space of continuous functions that are
piecewise linear over the triangulation $\mathcal{T}_{h}$ with $V_h
\subset V$.

\begin{assumption}
\label{assumption2}
 \textbf{(Regularity of the domain $\Omega$ and the space grid)}

 We assume that $\Omega$  has a smooth boundary or is a convex polygon
 and  that the maximal length  $h$ of the elements of
 $\mathcal{T}_{h}$ satisfy the usual  regularity assumptions on the
 triangulation i.e. for $r=1,2$
\begin{eqnarray}
\label{regulard}
  \underset{\chi \in V_{h}}{\inf}\left(\Vert v-\chi \Vert +h \Vert
    \nabla(v-\chi)\Vert\right) \leq  C h^{r} \Vert
  v\Vert_{H^{r}(\Omega)}, \qquad v\in V\cap
  H^{r}(\Omega). 
\end{eqnarray}
\end{assumption}
This inequality is sometimes called the Bramble and
Hilbert inequality, see \cite{lions} or \cite{Vidar}. It follows 
that 
\begin{equation}
 \Vert P_{h}v -v \Vert  \leq C h^{r}\Vert v \Vert_{H^{r}(\Omega)}
 \qquad  \forall v\in V\cap H^{r}(\Omega),\quad 
 r=1,2.
\label{eq:feError}
\end{equation}

\begin{assumption}
 \label{assumption4}
\textbf{(Nonlinearity)}

Let $\mathcal{V}$ be a separable Banach space such that
$\mathcal{D}((-A)^{1/2}) \subset \mathcal{V} \subset H=L^{2}(\Omega) $
continuously. 
We assume that there exist a positive constant $L> 0$  such  that
the Nemytskii  operator $F$ satisfies one of the following


(a) $F: \mathcal{V} \rightarrow \mathcal{V} $ is twice continuously Fr\'{e}chet  differentiable mapping with
$$
\Vert  F'(v)w \Vert  \leq L \Vert w\Vert,
\qquad \Vert F'(v)\Vert_{L(\mathcal{V})}\leq L,
\qquad \Vert F''(v)\Vert_{L^{(2)}(\mathcal{V})}\leq L$$
and 
$$
\Vert (F'(u))^{*}\Vert_{L(\mathcal{D}((-A)^{1/2}))}  \leq L(1+ \Vert
u\Vert_{\mathcal{D}((-A)^{1/2})}) \qquad \forall \,v,w \in 
\mathcal{V},\quad u \in \mathcal{D}((-A)^{1/2}),
$$ 
where $(F'(u))^{*}$ is the adjoint of $F'(u)$ defined by
\begin{eqnarray*}
 ((F'(u))^{*}v,w)=(v,F'(u)w)\;\;\ \forall \,v,w \in H=L^{2}(\Omega).
\end{eqnarray*}
As a consequence 
\begin{eqnarray*}
 \Vert F(X)- F(Y)\Vert \leq L \Vert X- Y\Vert \qquad \forall X,
 Y  \in H, 
\end{eqnarray*}
and $\forall X \in H=L^{2}(\Omega)$
\begin{eqnarray*}
 \Vert F(X) \Vert \leq   \Vert F (0)\Vert + \Vert F (X) - F (0)\Vert
 \leq \Vert F (0)\Vert + L \Vert X\Vert \leq C( \Vert F (0)\Vert
 +\Vert X \Vert ).
\end{eqnarray*}
(b) $ F$ is globally Lipschitz continuous from 
$(H^{1}(\Omega),\Vert. \Vert_{H^{1}(\Omega)})$ to $(H=L^{2}(\Omega),
\Vert. \Vert)$  then 
\begin{eqnarray*}
  \Vert F(X)- F(Y)\Vert \leq L \Vert X- Y\Vert_{H^{1}(\Omega)} \qquad
  \forall X, Y  \in H^{1}(\Omega). 
\end{eqnarray*}
\end{assumption}
We assume that the function  $F$  is  defined in $L^{2}(\Omega)$,
although in general $F$ may be defined in any Hilbert space.
The possible choice of $\mathcal{V}$ can be $H$, $H^{1}(\Omega)$  or
the $\mathbb{R}-$ Banach space of continuous functions from $[0,T]$ to
$H$ denoted by $C([0,T],H)$ if $d=1$. 

We now turn our attention to the noise term and introduce spaces and
notation that we need to define the $Q$-Wiener process.
Denoting by $\mathcal{L}(H)$ the Banach algebra of bounded linear
operators on $H$ with the usual norm. We recall that an operator $T
\in \mathcal{L}(H)$ is Hilbert-Schmidt if
\begin{eqnarray*}
 \Vert T\Vert_{HS}^{2}:=\underset{i\in \mathbb{N}^{d}}{\sum}\Vert T e_{i}
 \Vert ^2 < \infty. 
\end{eqnarray*}
If we denote the space of Hilbert-Schmidt operator from 
$Q^{1/2}(H)$ to $H$ by $L_{2}^{0}:=  HS(Q^{1/2}(H),H)$ i.e   
\begin{eqnarray*}
  L_{2}^{0}= \left\lbrace \varphi \in \mathcal{L}(H) :\underset{i \in
      \mathbb{N}^{d}}{\sum}\Vert \varphi \,Q^{1/2} e_{i} \Vert^{2}< \infty
  \right\rbrace, 
\end{eqnarray*}
the corresponding norm $\Vert . \Vert_{L_{2}^{0}}$ by
\begin{eqnarray*}
  \Vert \varphi \,\Vert_{L_{2}^{0}} := \Vert \varphi
  Q^{1/2}\Vert_{HS}=\left( \underset{i \in \mathbb{N}^{d}}{\sum}\Vert
    \varphi Q^{1/2} e_{i} \Vert^{2}\right)^{1/2}.  
\end{eqnarray*}
Let $\varphi(\omega)$ be a process such that for every sample
$\omega$, $\varphi(\omega) \in L_{2}^{0}$. Then we have the following
equality 
\begin{eqnarray*}
 \textbf{E} \Vert \int_{0}^{t}\varphi dW \Vert^{2}=\int_{0}^{t}
 \textbf{E}\Vert \varphi \Vert_{L_{2}^{0}}^{2}ds=\int_{0}^{t}
 \textbf{E}\Vert \varphi Q^{1/2} \Vert_{HS}ds,
\end{eqnarray*}
using Ito's isometry \cite{DaPZ}.
We assume sufficient regularity of the noise for the existence of a
mild solution and to project the noise into the finite element space
$V_h$. To be specific we assume the noise is in either in $H^1$ or
$H^2$ in space. 

\begin{assumption}
\label{assumption3}
\textbf{(Regularity of the noise) }
For all $t \in [ 0,T]$ we assume  that $O(t)$ is an adapted stochastic
process to the filtration $(\mathcal{F}_{t})_{t\geq 0 }$ with  
continuous sample paths such that $O(t_{2})-S(t_{2}-t_{1})O(t_{1}),
\;0\leq t_{1}<t_{2} \leq T$ is independent of $\mathcal{F}_{t_{1}}$  
Let $\mathcal{V}$ be a separable Banach space such that
$\mathcal{D}((-A)^{1/2}) \subset \mathcal{V} \subset H=L^{2}(\Omega) $
continuously. For some $\theta\in(0,1/2]$  
\begin{eqnarray*}
 O(t) \in \mathcal{D}((-A)^{r/2})&= &\HH \cap H^{r}(\Omega),\qquad  r=1,2\\
\textbf{E}\left(\Vert O(t_{2})-O(t_{1})\Vert_{\mathcal{V}}^{4}\right) &\leq&
C(t_{2}-t_{1})^{4 \theta}.
\end{eqnarray*}
\end{assumption}

Using the equivalence of norms, we have that
\begin{eqnarray*}
  O(t) \in \mathcal{D}((-A)^{r/2}) \quad \forall t \in  [ 0,T]
  \Leftrightarrow \Vert(- A)^{r/2} Q^{1/2}\Vert_{HS} <\infty \quad  r=1,2.
\end{eqnarray*}

\subsection{Main results}
Throughout the paper we let $N$ be the number of terms of truncated
noise and $\mathcal{I}_{N}=\left\lbrace 1,2,...,N \right\rbrace^{d}$ and
take $t_m=m\Dt \in (0,T]$, where $T=M\Dt$ for $m,M\in\N$. We take $C$
to be a constant that may depend on $T$ and other parameters but not
on $\Dt$, $N$ or $h$. We also assume that when initial data $X_{0}  \in
\mathcal{D}((-A)^{\gamma})$  then $  \mathbf{E}\Vert
(-A)^{\gamma}X_{0}\Vert^{l} < \infty$, $l=2,4$ and $0\leq \gamma <1$.

Our first result is a strong convergence result in $L^2$ when the 
non-linearity satisfies the Lipschitz condition of
\assref{assumption4} (a) with scheme (\textbf{SETD1}). 
This is, for example, the case of reaction--diffusion SPDEs. 
\begin{theorem}
\label{th1}
Suppose that Assumptions \ref{assumption1}, \ref{assumption2}, 
\ref{assumption4}(a) and \ref{assumption3} are satisfied with $r=1,2$.  
Let  $X(t_m)$  be the mild  solution of equation \eqref{adr} represented
by (\ref{eq1}) and $X_m^{h}$ be the numerical
approximation through \eqref{new} (\textbf{SETD1}). Let  $ 0< \gamma <
1$ and  set $\sigma = \min(2\theta,\gamma)$ and let $\theta \in
(0,1/2]$ be defined as in \assref{assumption3}.  

If  $X_{0} \in \mathcal{D}((-A)^{\gamma})$ then
\begin{eqnarray*}
 \left(\textbf{E}\Vert X(t_m)-X_{m}^{h}\Vert^{2}\right)^{1/2}& \leq
 &C \left( t_{m}^{-1/2} h^{r}+\Delta t^{\sigma} +\left(\underset { j
       \in \mathbb{N}^{d} \backslash 
       \mathcal{I}_{N}} {\inf}  \lambda_{j}\right)^{-r/2}\right).
\end{eqnarray*}
If $X_{0} \in \mathcal{D}(-A) = \HH \cap H^{2}(\Omega)$ then 
\begin{eqnarray*}
 \left(\textbf{E}\Vert X(t_m)-X_{m}^{h}\Vert^{2}\right)^{1/2}
 \leq C \left(h^{r}+\Delta t^{2 \theta }+\left(\underset { j \in
       \mathbb{N}^{d} \backslash 
       \mathcal{I}_{N}} {\inf}  \lambda_{j}\right)^{-r/2}\right).
\end{eqnarray*}
\end{theorem}
Our first result  for scheme (\textbf{SETD0}) is a strong convergence
result in $L^2$ when the non-linearity satisfies the Lipschitz
condition of \assref{assumption4} (a).  

\begin{theorem}
\label{th3}
Suppose that Assumptions \ref{assumption1}, \ref{assumption2}, 
\ref{assumption4}(a) and \ref{assumption3} are satisfied with $r=1,2$.  
Let  $X(t_m)$  be the mild  solution of equation \eqref{adr} represented
by (\ref{eq1}) and $Y_m^{h}$ be the numerical
approximation through \eqref{new0}  (\textbf{SETD0}). Let  $ 0< \gamma
< 1$ and  set $\sigma = \min(2\theta,\gamma)$ and let $\theta \in
(0,1/2]$ be defined as in \assref{assumption3}.  

If  $X_{0} \in \mathcal{D}((-A)^{\gamma})$ then
\begin{eqnarray*}
 \left(\textbf{E}\Vert X(t_m)-Y_{m}^{h}\Vert^{2}\right)^{1/2}& \leq
 &C \left( t_{m}^{-1/2} h^{r}+\Delta t^{\sigma} +\Delta t \left|\ln (\Delta t)
   \right|+ \left(\underset { j \in \mathbb{N}^{d} \backslash
       \mathcal{I}_{N}} {\inf}  \lambda_{j}\right)^{-r/2}\right).
\end{eqnarray*}
If $X_{0} \in \mathcal{D}(-A) = \HH \cap H^{2}(\Omega)$ then 
\begin{eqnarray*}
 \left(\textbf{E}\Vert X(t_m)-Y_{m}^{h}\Vert^{2}\right)^{1/2}
 \leq C \left(h^{r}+\Delta t^{2 \theta }+\Delta t \left|\ln (\Delta t)
   \right|+\left(\underset { j \in \mathbb{N}^{d} \backslash
       \mathcal{I}_{N}} {\inf}  \lambda_{j}\right)^{-r/2}\right).
\end{eqnarray*}
\end{theorem}
For convergence in the mean square  $H^{1}(\Omega)$ 
norm where the non-linearity satisfies the Lipschitz condition
from  $L^{2}(\Omega)$ norm to $H^{1}(\Omega)$ ( \assref{assumption4} (b)) we can state results for (\textbf{SETD1}) and
(\textbf{SETD0}) together.
\begin{theorem}
\label{th2}
Suppose that Assumptions
\ref{assumption1}, \ref{assumption2}, \ref{assumption4}(b), \ref{assumption3}
(with $r=2 $) are satisfied and  $F(X)\in V$ with  $\mathbf{E}\left(\underset{0\leq s\leq
      T}{\sup} \Vert
    F(X(s))\Vert_{H^{1}(\Omega)}\right)^{2} < \infty $. 
Let $X$ be the solution mild of equation (\ref{adr}) represented by equation
 (\ref{eq1}) and $\zeta_m^{h}$ be the numerical
approximations through scheme \eqref{new}  or \eqref{new0} (
$\zeta_m^{h}=X_m^{h} $  
for scheme (\textbf{SETD1}) and $\zeta_m^{h}=Y_m^{h} $ for scheme
(\textbf{SETD0})). 
 Let  $ 0\leq \gamma < 1$.
Then we have the following:

If $ X_{0}\in \mathcal{D}((-A)^{(1+\gamma)/2})$ then 
\begin{eqnarray*}
(\textbf{E} \Vert X(t_m)-\zeta_{m}^{h}\Vert_{H^{1}(\Omega)}^{2})^{1/2}&\leq& C
\left((t_{m}^{-1/2}\,h +\Delta t^{\gamma/2} )  + \left(
    \underset { j \in \mathbb{N}^{d} \backslash \mathcal{I}_{N}}
    {\inf} \lambda_{j}\right)^{-1/2}\right). 
\end{eqnarray*}
If $ X_{0}\in  \mathcal{D}((-A))$  then
\begin{eqnarray*}
(\textbf{E} \Vert
X(t_m)-\zeta_{m}^{h}\Vert_{H^{1}(\Omega)}^{2})^{1/2}&\leq& C \left ((h 
  +\Delta t^{1/2-\epsilon} )  + \left( \underset { j \in
      \mathbb{N}^{d} \backslash \mathcal{I}_{N}} {\inf}
    \lambda_{j}\right)^{-1/2}\right). 
\end{eqnarray*}
for very small $\epsilon \in (0,1/2)$.
\end{theorem}
We note that this theorem covers the case of
advection-diffusion-reaction SPDEs, such as that arising in our
example from porous media.

We remark that if we denote by  $N_{h}$ the numbers of
vertices in the finite elements mesh then it is well known (see for 
example \cite{Yn:04}) that  if $N \geq N_{h}\ $ then 
$$
  \left( \underset { j \in \mathbb{N}^{d} \backslash \mathcal{I}_{N}}
    {\inf} \lambda_{j}\right )^{-1} \leq C h^{2}
\qquad \text{and} \qquad 
  \left( \underset { j \in \mathbb{N}^{d} \backslash \mathcal{I}_{N}}
    {\inf} \lambda_{j}\right )^{-1/2} \leq C h.
$$
As a consequence the estimates in Theorem \ref{th1}, Theorem
\ref{th3}  and Theorem \ref{th2} can be expressed in function of $h$
and $\Delta t $ only and it is the error from the finite element
approximation that dominates. If $N\leq N_h$ then it is the error from
the projection $P_N$ of the noise onto a finite number of modes that
dominates. 

From \thmref{th2} we also get an estimate in the root mean square
$L^2(\Omega)$  norm in the case that the nonlinear function $F$
satisfies \assref{assumption4} (b). We cannot do the proof directly in
$ L^2(\Omega)$ due to the Lipschitz condition in \assref{assumption4}
(b). 
Simulations  for  \thmref{th2} will be do in  $
L^2(\Omega)$ since the discrete $L^2(\Omega)$ norm is more  easy to use in all type of boundary conditions.

Finally if we compare these theorems to those in \cite{GTambue} for a
modified semi-implicit Euler-Maruyama method then we see that 
using the exponential based integrators we have weaker conditions on
the initial data and in particular the scheme \SETD has better convergence
properties.

\section{Proofs of main results}
\label{sec:proofs}

\subsection{Preparatory results}

We start by examining the deterministic linear problem.
Find $u \in V$ such that such that  
\begin{eqnarray}
\label{homog}
u'=Au \qquad  
\text{given} \quad u(0)=v.
\end{eqnarray}
The corresponding semi-discretization in space is : find $u_{h} \in
V_{h}$ such that  
$$u_{h}'=A_{h}u_{h}$$ 
where $u_{h}^{0}=P_{h}v$.
Define the operator 
\begin{eqnarray}
\label{form1}
T_{h}(t) :=  S(t)-S_{h}(t) P_{h} = e^{tA} - e^{tA_h}P_h
\end{eqnarray} 
so that $u(t)-u_{h}(t)= T_{h}(t) v$. 
\begin{lemma}
\label{lemme1}
The following estimates hold on the semi-discrete approximation of
\eqref{homog}
\begin{eqnarray}
\label{form2}
\Vert u(t)-u_{h}(t)\Vert &=&\Vert T_{h}(t) v\Vert \leq C t^{-1/2}\;
h^{2}\Vert v \Vert \;\;\;\;\;\text{if}\;\;\; v \in H=L^{2}(\Omega),\\ 
\label{form3}
\Vert u(t)-u_{h}(t)\Vert &=&\Vert T_{h}(t) v\Vert \leq C h^{2}\Vert v
\Vert_{H^{2}(\Omega)}\;\;\;\text{if}\;\;\; v \in \mathcal{D}((-A)),\\ 
\label{form4}
\Vert u(t)-u_{h}(t)\Vert_{H^{1}(\Omega)} &=&\Vert T_{h}(t) v\Vert_{H^{1}(\Omega)} \leq
C h t^{-1/2}\Vert v \Vert_{H^{1}(\Omega)}\;\;\;\text{if}\;\;\; v \in
V,\\ 
\label{form5}
\Vert u(t)-u_{h}(t)\Vert_{H^{1}(\Omega)} &=&\Vert T_{h}(t) v\Vert_{H^{1}(\Omega)} \leq
C h \Vert v \Vert_{H^{2}(\Omega)}\;\;\;\text{if}\;\;\; v \in
\mathcal{D}(-A). 
\end{eqnarray}
\begin{proof}
The proof of the estimates (\ref{form2}) and (\ref{form3}) can be
found in \cite{Vidar} with Dirichlet boundary conditions.  
The same proof can be generalized easily to Robin or mixed boundary
conditions, incorporating the extra term from the boundary with the
bilinear form.
Estimates (\ref{form2})- (\ref{form5}) are the special case of the proof
of Theorem 3.1 in \cite{Stig} where the nonlinearity is taken to be
zero.
For our case
\begin{eqnarray*}
 u(t)=S(t)v,
\end{eqnarray*}
and we have the  following estimates for  $t \in (0,T]$ 
\begin{eqnarray*}
  \Vert u(t)\Vert_{H^{s}(\Omega)} &\leq &  C  t^{-(s-1)/2}\Vert  v
  \Vert_{H^{1}(\Omega)}\;\;\;\;\;\text{if}\;\;\; v \in
  \mathbb{H}\;\;\;\;\; s= 1,2, \\ 
 \Vert u(t)\Vert_{H^{2}(\Omega)}  & \leq &  C  \Vert  v
 \Vert_{H^{2}(\Omega)}\;\;\;\;\;\text{if}\;\;\; v \in
 \mathcal{D}(-A),\\ 
\Vert u_{t}(t)\Vert_{H^{s}(\Omega)}&\leq & C t^{-1-(s-1)/2} \Vert v
\Vert_{H^{1}(\Omega)} \;\;\;\;\;\text{if}\;\;\; v \in \mathbb{H}\;\;
\;  s= 0,1, \\ 
\Vert u_{t}(t)\Vert_{H^{2}(\Omega)}&\leq & C t^{-s/2} \Vert v
\Vert_{H^{2}(\Omega)} \;\;\;\;\;\text{if}\;\;\; v \in
\mathcal{D}(-A)\;\; \;  s= 0,1.
\end{eqnarray*}
Using these in the proof of \cite[Theorem 3.2]{Stig} gives the result.
\end{proof}
\end{lemma}

We now consider the SPDE \eqref{adr}
\begin{lemma}
\label{lemme2}
Let $X$ be the  mild solution  of (\ref{adr}) given  in (\ref{eq1}),
let  $0 \leq \gamma < 1 $ and $t_{1}, t_{2} \in [0,T],\quad t_{1}<
t_{2}$.

(i) If $X_{0} \in  \mathcal{D}((-A)^{\gamma}$), $\Vert
(-A)^{\alpha/2 }Q^{1/2}\Vert_{HS}<\infty$ with $0 \leq \alpha \leq 2$
and suppose $F$ satisfies  \assref{assumption4} (a). Set  $\sigma=
\min( \gamma,1/2,\alpha/2)$ then  
\begin{eqnarray*}
\textbf{E}\Vert X(t_{2})- X(t_{1}) \Vert^{2} &\leq&  C
(t_{2}-t_{1})^{2\sigma} \left(\textbf{E} \Vert
  X_{0}\Vert_{\gamma}^{2}+\textbf{E}\left(\underset{0\leq s\leq
      T}{\sup} \left(\Vert F(0)\Vert+ \Vert X(s)\Vert
    \right)\right)^{2}+1 \right).
\end{eqnarray*}
Furthermore 
\begin{eqnarray*}
 \mathbf{E}\Vert (X(t_{2})-O(t_{2}))-( X(t_{1})-O(t_{1})) \Vert^{2}
 &\leq& C (t_{2}-t_{1})^{2 \gamma}\;\;\;\;0 \leq \gamma \leq 1 . 
\end{eqnarray*}
(ii) If $X_{0} \in  \mathcal{D}((-A)^{(\gamma+1)/2})$, $\Vert
(-A)^{1/2 }Q^{1/2}\Vert_{HS}<\infty$ and $ F(X) \in H^{1}(\Omega)$ with $\mathbf{E}\left(\underset{0\leq s\leq
      T}{\sup} \Vert
    F(X(s))\Vert_{H^{1}(\Omega)}\right)^{2} < \infty $
then 
\begin{eqnarray*}
\mathbf{E}\Vert X(t_{2})- X(t_{1}) \Vert^{2} \leq C
(t_{2}-t_{1})^{\gamma}\left(\textbf{E} \Vert
  X_{0}\Vert_{(\gamma+1)/2}^{2}+\mathbf{E}\left(\underset{0\leq s\leq
      T}{\sup} \Vert
    F(X(s))\Vert_{H^{1}(\Omega)}\right)^{2}+1\right). 
\end{eqnarray*}
\end{lemma}
\begin{remark}
 Before doing the proof it is important to notice that if $X_{0}\in \mathcal{D}((-A)^{\gamma})$ with $\mathbf{E} \Vert(-A)^{\gamma}X_{0})\Vert^{l} < \infty,\, l=2,4 $,
 Assumption 1-4 ensure the existence of the unique solution $X \in \mathcal{D}((-A)^{\gamma})$ 
such that
$$ \mathbf{E}\left(\underset{0\leq s\leq T}{\sup} \Vert
    (-A)^{\gamma}X(s))\Vert^{l}\right)< \infty. $$

In general  if $X_{0}\in \mathcal{D}((-A)^{\gamma})$ and $O(t) \in \mathcal{D}((-A)^{\alpha})$ then $X \in \mathcal{D}((-A)^{min(\gamma,\alpha)})$ with 
$$ \mathbf{E}\left(\underset{0\leq s\leq T}{\sup} \Vert
    (-A)^{min(\gamma,\alpha)}X(s)\Vert^{l}\right)< \infty .$$
More  information about properties  of the solution of the SPDE (\ref{adr}) can be found in \cite{Jentzen4}.
\end{remark}
\begin{proof}
The first claim of part (i) of the Lemma can be found in
\cite{GTambue} and so we prove the second part of (i).  Consider the
difference 
\begin{eqnarray*}
\lefteqn{ ( X(t_{2}) +O(t_{2})- (X(t_{1})+O(t_{1}) ))}\\
&=&  \left(S(t_{2})-S(t_{1})\right)X_{0}+\left(
  \int_{0}^{t_{2}}S(t_{2}-s)F(X(s))ds -
  \int_{0}^{t_{1}}S(t_{1}-s)F(X(s))ds\right) \\    
&=& I +II 
\end{eqnarray*}
so that 
$$ \textbf{E}  \Vert ( X(t_{2}) +O(t_{2}))- (X(t_{1})+O(t_{1}))\Vert^{2} 
\leq 2 (\textbf{E}  \Vert I \Vert^{2} + \textbf{E} \Vert II \Vert^{2}) .$$ 
We estimate each of the terms $I, II$ . For  $\;0 \leq \gamma \leq 1 $, using
\propref{prop1} yields
\begin{eqnarray*}
 \Vert I \Vert
&=&\Vert S(t_{1})(-A)^{-\gamma}(\text{I}-S(t_{2}-t_{1})) (-A)^{\gamma}
X_{0} \Vert  
\quad \leq \quad   C (t_{2}-t_{1})^{\gamma} \Vert X_{0} \Vert_{\gamma}.
\end{eqnarray*}
Then $ \textbf{E} \Vert I \Vert^{2} \leq  C (t_{2}-t_{1})^{2
  \gamma}\textbf{E} \Vert X_{0} \Vert_{\gamma}^{2}$.  
For the term $II$, we have 
\begin{eqnarray*}
II&= &\int_{0}^{t_{1}}(S(t_{2}-s)-S(t_{1}-s))F(X(s))ds +\int_{t_{1}}^{t_{2}}S(t_{2}-s)F(X(s))ds\\
  &=& II_{1}+II_{2}.
 \end{eqnarray*}
We now estimate each term $II_1$ and $II_2$.  For $ \textbf{E} \Vert
II_2 \Vert ^{2}$ boundedness of $S$ gives
\begin{eqnarray*}
  \textbf{E} \Vert II_2 \Vert ^{2} &\leq& \left(\int_{t_{1}}^{t_{2}}
    \textbf{E} \Vert S(t_{2}-s)F(X(s)) \Vert ds\right)^{2}\\ 
&\leq&  C \left( t_{2}-t_{1}\right)^{2}  \textbf{E}
\left(\underset{0\leq s\leq T}{\sup} \Vert F(X(s))\Vert\right)^{2}. 
\end{eqnarray*}
For $ \textbf{E} \Vert II_1 \Vert ^{2}$ we have 
\begin{eqnarray*}
  II_{1}&=&\int_{0}^{t_{1}}(S(t_{2}-s)-S(t_{1}-s))F(X(s))ds\\
  &=& \int_{0}^{t_{1}}(S(t_{2}-s)-S(t_{1}-s))
  \left( F(X(s))-F(X(t_{1}))\right)ds \\ 
  &+& \int_{0}^{t_{1}}(S(t_{2}-s)-S(t_{1}-s))F(X(t_{1}))ds\\
  &=& II_{11}+ II_{12}.
\end{eqnarray*}
Using the Lipschitz condition in \assref{assumption4} (a)
with the first claim of (i) yields   
\begin{eqnarray*}
 \textbf{E} \Vert II_{11} \Vert ^{2} &\leq&\left(\int_{0}^{t_{1}}(
   S(t_{2}-s)-S(t_{1}-s))\textbf{E}\Vert (X(s)-X(t_{1})\Vert
   ds\right)^{2}\\ 
 &\leq& C \left((t_{2}-t_{1}) \int_{0}^{t_{1}} (t_{1}-s)^{\sigma-1}
   ds\right)^{2}\\ 
 &\leq & C \left(t_{2}-t_{1}\right)^{2}.
\end{eqnarray*}
\assref{assumption4} (a) gives
 \begin{eqnarray*}
   \left(\textbf{E} \Vert II_{12} \Vert ^{2}\right)^{1/2} &\leq&
   \left(\textbf{E}\Vert F(X(t_{1})\Vert^{2} \right)^{1/2}
   \Vert\int_{0}^{t_{1}}( S(t_{2}-s)-S(t_{1}-s)) ds\Vert\\ 
   &\leq& C \Vert \int_{0}^{t_{1}} S(t_{2}-s)-S(t_{1}-s)ds \Vert 
\end{eqnarray*}
Using the fact that $S$ is bounded we find
 \begin{eqnarray*}
   \left(\textbf{E} \Vert II_{12} \Vert ^{2}\right)^{1/2} 
   & \leq& C  \Vert S(t_{1})\int_{0}^{t_{1}} \left( S(t_{2}-
     t_{1}-s)-S(-s) \right) ds \Vert\\ 
   &\leq & C \Vert\int_{0}^{t_{1}} \left( S(t_{2}- t_{1}+s)-S(s)
   \right) ds \Vert \\ 
   &= & C \Vert\int_{t_{2}- t_{1}}^{t_{2}} S(s) ds - \int_{0}^{t_{1}}
   S(s) ds \Vert \\ 
   &= & C \Vert\int_{t_{2}- t_{1}}^{t_{1}} S(s) ds +
   \int_{t_{1}}^{t_{2}} S(s) ds - \int_{0}^{t_{1}} S(s)ds \Vert \\ 
   &= & C \Vert  \int_{t_{1}}^{t_{2}} S(s) ds - \int_{0}^{t_{2}-t_{1}}
   S(s)ds \Vert \\ 
   &\leq& C (t_{2}-t_{1}).
\end{eqnarray*}  
Combining the previous estimates ends the proof of the second claim of
(i). 

We now prove part (ii) of the lemma. Consider the difference  
\begin{eqnarray*}
\lefteqn{  X(t_{2})- X(t_{1}) }\\
&=& \left(S(t_{2})-S(t_{1})\right)X_{0}+
  \left(\int_{0}^{t_{2}}S(t_{2}-s)
    F(X(s))ds-\int_{0}^{t_{1}}S(t_{1}-s)F(X(s))ds \right)\\  
  & & +\left(\int_{0}^{t_{2}}S(t_{2}-s)dW(s) 
    - \int_{0}^{t_{1}}S(t_{1}-s)dW(s)\right)\\  
  &=& I +II +III
\end{eqnarray*}
and then
\begin{eqnarray*}
 \textbf{E}  \Vert X(t_{2})- X(t_{1})\Vert_{H^{1}(\Omega)}^{2} \leq 3
 \left(\textbf{E}  \Vert I \Vert_{H^{1}(\Omega)}^{2} + \textbf{E} \Vert II
 \Vert_{H^{1}(\Omega)}^{2} + \textbf{E} \Vert III
 \Vert_{H^{1}(\Omega)}^{2}\right). 
\end{eqnarray*}
Let us estimate the terms $I, II$ and $III$ and we start with $I$. 
If $X_{0} \in \mathcal{D}((-A)^{(\gamma+1)/2})$ using
\propref{prop1} yields 
\begin{eqnarray*}
 \Vert I \Vert_{H^{1}(\Omega)}&=& \Vert (-A)^{1/2}
 S(t_{1})(\text{I}-S(t_{2}-t_{1})) X_{0} \Vert \\ 
&=& \Vert(-A)^{1/2} S(t_{1})(\text{I}-S(t_{2}-t_{1}))
(-A)^{-\gamma/2}(-A)^{\gamma/2} X_{0} \Vert \\ 
&=&\Vert S(t_{1})(-A)^{-\gamma/2}(\text{I}-S(t_{2}-t_{1}))
(-A)^{(\gamma+1)/2} X_{0} \Vert \\ 
&\leq&  C (t_{2}-t_{1})^{\gamma/2} \Vert X_{0} \Vert_{(\gamma+1)}. 
\end{eqnarray*}
Then
\begin{eqnarray*}
 \textbf{E} \Vert I \Vert_{H^{1}(\Omega}^{2} \leq  C (t_{2}-t_{1})^{\gamma} \Vert
 X_{0} \Vert_{(\gamma+1)}^{2}. 
\end{eqnarray*}
For the term $II$, we have 
\begin{eqnarray*}
II&= &\int_{0}^{t_{1}}(S(t_{2}-s)-S(t_{1}-s))F(X(s))ds
+\int_{t_{1}}^{t_{2}}S(t_{2}-s)F(X(s))ds\\ 
&=& II_{1}+II_{2}.
\end{eqnarray*}
We now estimate each term above. 
Using the fact that in $\mathcal{D}((-A)^{1/2})$ we have the
equivalency of norm $\Vert.\Vert_{H^{1}(\Omega}\equiv \Vert (-A)^{1/2}
.\Vert$,  
we have  
\begin{eqnarray}
  \label{ineq}
  \Vert S(t)\Vert_{L(H^{1}(\Omega))}\leq \Vert (-A)^{1/2} S(t)\Vert
\end{eqnarray}
where $\Vert S(t)\Vert_{L(H^{1}(\Omega))}$ is the norm of the semigroup
viewed as a bounded operator in $H^{1}(\Omega)$. We also have the
similar relationship for the operator $S(t_{1})-S(t_{2})$ with
$t_{1}$, $t_{2} \in [0,T]$. 

Using similar inequality as (\ref{ineq}) yields
\begin{eqnarray*}
 \textbf{E} \Vert II_{1} \Vert_{H^{1}(\Omega)}^{2} &=& \textbf{E}\Vert
 \int_{0}^{t_{1}}(S(t_{2}-s)-S(t_{1}-s))F(X(s))ds
 \Vert_{H^{1}(\Omega)}^{2}\\ 
  &\leq&  \textbf{E} \left(\int_{0}^{t_{1}}\Vert
    (S(t_{2}-s)-S(t_{1}-s)) F(X(s))\Vert_{H^{1}(\Omega)} ds\right)^{2}
  \;\;\\ 
&\leq&  \left(\int_{0}^{t_{1}}\Vert
  (S(t_{2}-s)-S(t_{1}-s))\Vert_{H^{1}(\Omega)} ds\right)^{2} \,
\textbf{E} \left(\underset{0\leq s\leq T}{\sup} \Vert
  F(X(s))\Vert_{H^{1}(\Omega)}\right)^{2}\\ 
&\leq&  \left(\int_{0}^{t_{1}}\Vert
  (-A)^{1/2}(S(t_{2}-s)-S(t_{1}-s))\Vert ds\right)^{2} \,\textbf{E}
\left(\underset{0\leq s\leq T}{\sup} \Vert
  F(X(s))\Vert_{H^{1}(\Omega)}\right)^{2}.
\end{eqnarray*}
For $\epsilon \in (0,1)$ small enough, we have 
\begin{eqnarray*}
\lefteqn{ \textbf{E} \Vert II_{1} \Vert_{H^{1}(\Omega)}^{2}}\\
 &\leq&
 \left(\int_{0}^{t_{1}}\Vert (-A)^{(1-\epsilon)/2 } S(t_{1}-s)
   (-A)^{(\epsilon-1)/2 }((\text{I}-S(t_{2}-t_{1})))\Vert ds
 \right)^{2} \, \textbf{E}\left(\underset{0\leq s\leq T}{\sup} \Vert
   F(X(s))\Vert_{H^{1}(\Omega)}\right)^{2}\\ 
  &\leq & C
  (t_{2}-t_{1})^{1-\epsilon}\left(\int_{0}^{t_{1}}(t_{1}-s)^{(1-\epsilon)/2}ds  \right)^{2} \, \textbf{E}\left(\underset{0\leq s\leq T}{\sup} \Vert F(X(s))\Vert_{H^{1}(\Omega)}\right)^{2}\\    
&\leq & C  (t_{2}-t_{1})^{1-\epsilon} \,
\textbf{E}\left(\underset{0\leq s\leq T}{\sup} \Vert
  F(X(s))\Vert_{H^{1}(\Omega)}\right)^{2}. 
\end{eqnarray*}
We also have using \propref{prop1}
\begin{eqnarray*}
 \textbf{E} \Vert II_{2} \Vert_{H^{1}(\Omega)}^{2} &=&\textbf{E} \Vert
 \int_{t_{1}}^{t_{2}}S(t_{2}-s)F(X(s))ds\Vert_{H^{1}(\Omega)}^{2}\\ 
&\leq& \textbf{E} \left(\int_{t_{1}}^{t_{2}}\Vert S(t_{2}-s)F(X(s))
  \Vert_{H^{1}(\Omega)} ds \right)^{2}\\ 
&\leq& \textbf{E} \left(\int_{t_{1}}^{t_{2}}\Vert
  S(t_{2}-s)\Vert_{L(H^{1}(\Omega))}\Vert
  F(X(s))\Vert_{H^{1}(\Omega)}ds\right)^{2}\\ 
&\leq& \textbf{E} \left(\int_{t_{1}}^{t_{2}}\Vert
  (-A)^{1/2}(S(t_{2}-s))\Vert\Vert
  F(X(s))\Vert_{H^{1}(\Omega)}ds\right)^{2}\\ 
&\leq& \left( \int_{t_{1}}^{t_{2}}(t_{2}-s)^{-1/2} ds \right)^{2}
\textbf{E}\left(\underset{0\leq s\leq T}{\sup}\Vert
  F(X(s))\Vert_{H^{1}(\Omega)}\right)^{2}\\ 
&\leq& C (t_{2}-t_{1})\textbf{E}\left(\underset{0\leq s\leq
    T}{\sup}\Vert F(X(s))\Vert_{H^{1}(\Omega)}\right)^{2}. 
\end{eqnarray*}
Hence, if  $ F(X) \in H^{1}(\Omega) $ with $\mathbf{E}\left(\underset{0\leq s\leq
      T}{\sup} \Vert
    F(X(s))\Vert_{H^{1}(\Omega)}\right)^{2}< \infty $, we have 
\begin{eqnarray*}
\textbf{E} \Vert II\Vert^{2}\leq 2( \textbf{E} \Vert II_{1} \Vert^{2}
+\textbf{E} \Vert II_{2} \Vert^{2} )\leq C (t_{2}-t_{1})^{\gamma}
\textbf{E}\left(\underset{0\leq s\leq T}{\sup} \Vert
  F(X(s))\Vert_{H^{1}(\Omega)}\right)^{2}.  
\end{eqnarray*}
We also have for the term $III$  
\begin{eqnarray*}
  III & = & 
  \int_{0}^{t_{2}}S(t_{2}-s)dW(s)-\int_{0}^{t_{1}}S(t_{1}-s)dW(s)\\
&=&\int_{0}^{t_{1}}\left( S(t_{2}-s)-S(t_{1}-s)\right)dW(s)+
\int_{t_{1}}^{t_{2}}S(t_{2}-s)dW(s)\\  
  &=&III_{1} +III_{2}.
\end{eqnarray*}
The Ito isometry property yields 
\begin{eqnarray*}
  \textbf{E} \Vert III_{1} \Vert_{H^{1}(\Omega)}^{2}&=&
  \textbf{E}\Vert \int_{0}^{t_{1}}\left(
    S(t_{2}-s)-S(t_{1}-s)\right)dW(s)\Vert_{H^{1}(\Omega)}^{2}\\ 
  &\leq &  \int_{0}^{t_{1}}\textbf{E} \Vert
  (-A)^{1/2}\left(S(t_{2}-s)-S(t_{1}-s)\right)Q^{1/2}\Vert_{HS}^{2}
  ds\\ 
  &=&\int_{0}^{t_{1}}\textbf{E} \Vert\left(
    S(t_{2}-s)-S(t_{1}-s)\right)(-A^{1/2})Q^{1/2}\Vert_{HS}^{2} ds.
\end{eqnarray*}
Using \propref{prop1}, the fact that $S(t)$ is bounded  and
$\Vert (-A)^{1/2}Q^{1/2}\Vert_{HS}<\infty$  yields 
\begin{eqnarray*}
 \textbf{E} \Vert III_{1} \Vert_{H^{1}(\Omega)}^{2}&\leq & C
 \int_{0}^{t_{1}} \Vert ( S(t_{2}-s)-S(t_{1}-s)) \Vert^{2} ds \\ 
&= & C\int_{0}^{t_{1}} \Vert (-A)^{(1-\epsilon)/2 }
S(t_{1}-s)(-A)^{-(1-\epsilon)/2 }( \text{I}-S(t_{2}-t_{1}))\Vert^{2}
ds \\ 
&\leq & C (t_{2}-t_{1})^{1-\epsilon} \int_{0}^{t_{1}}
(t_{1}-s)^{-1+\epsilon}ds \\ 
&\leq& C (t_{2}-t_{1})^{1-\epsilon}.
\end{eqnarray*}
with  $\epsilon \in (0,1)$ small enough.
Let us estimate  $\textbf{E} \Vert III_{2} \Vert_{H^{1}(\Omega)} $.
The fact that $\Vert (-A)^{1/2}Q^{1/2}\Vert_{HS}<\infty$ yields 
\begin{eqnarray*}
 \textbf{E} \Vert III_{2} \Vert_{H^{1}(\Omega)}^{2}&=& \textbf{E}\Vert
 \int_{t_{1}}^{t_{2}}S(t_{2}-s))dW(s)\Vert_{H^{1}(\Omega)}^{2}\\ 
&\leq &  \int_{t_{1}}^{t_{2}} \Vert (-A)^{1/2}
S(t_{2}-s)Q^{1/2}\Vert_{HS}^{2} ds\\ 
&=&  \int_{t_{1}}^{t_{2}} \Vert S(t_{2}-s)(-A)^{1/2}
Q^{1/2}\Vert_{HS}^{2} ds\\ 
&\leq& C (t_{2}-t_{1}).
\end{eqnarray*}
Hence 
\begin{eqnarray*}
\textbf{E} \Vert III\Vert^{2}\leq 2( \textbf{E} \Vert III_{1}
\Vert^{2} +\textbf{E} \Vert III_{2} \Vert^{2} )\leq C
(t_{2}-t_{1})^{\gamma}.  
\end{eqnarray*}
Combining the estimates of  $\textbf{E} \Vert
I\Vert^{2},\textbf{E} \Vert II\Vert^{2}$ and $\textbf{E} \Vert
III\Vert^{2}$ ends the proof.  
\end{proof}
\begin{remark}
  If $\gamma\geq 1$ and with more regularity of the noise ($O(t) \in
  \mathcal{D}((-A)^{r}),\; r> 1/2 $) we have  
  \begin{eqnarray*}
    \textbf{E}\Vert X(t_{2})- X(t_{1}) \Vert^{2} \leq C
    (t_{2}-t_{1})^{1-\epsilon}
  \end{eqnarray*}
  for any $\epsilon\in(0,1)$.
  
  We can prove that we can take $\theta=1/2$ for  $\mathcal{V}=
  H^{1}(\Omega)$ or if $O(t) \in \mathcal{D}(-A)$. We have $\theta \neq 1/2$ 
  and close to $1/2$ if $O(t) \in \mathcal{D}((-A)^{1/2})$. These
  estimates follow those used to estimate  $III_{1}$ and  $III_{2}$
  in the proof of \lemref{lemme2}, see \cite{ATthesis}.
\end{remark}

\subsection{Proof of Theorem \ref{th1}}
\label{sec:th1}
The proof follows the same basic steps as in \cite{GTambue}, however here
the discrete semigroup is an exponential. As a consequence the
estimates are different and the proof here is simpler with fewer terms
to estimate.
Set 
\begin{eqnarray*}
 X(t_{m})&=&S(t_{m})X_{0}+\underset{k=0}{\sum^{m-1}}\int_{t_{k}}^{t_{k+1}}S(t_{m}-s)F(X(s))ds+O(t_{m})\\
          &=& \overline{X}(t_{m})+ O(t_{m}).
\end{eqnarray*}
Recall that by construction
\begin{eqnarray*}
X_m^h &=& e^{\Delta t A_{h}}X_{m-1}^{h}+\int_{0}^{\Delta t} e^{(\Delta t-
   s)A_{h}}P_{h}F(X_{m-1}^{h})ds + P_{h}\int_{t_{m-1}}^{t_{m}}
 e^{(t_{m}-s)A_{N}}d W^{N}(s)\\ 
&=& S_{h}(t_{m})
P_{h}X_{0}+\underset{k=0}{\sum^{m-1}}\left(\int_{t_{k}}^{t_{k+1}}
  S_{h}(t_{m}-s)P_{h}F(X_{k}^{h})ds +P_{h}\int_{t_{k}}^{t_{k+1}}
  S_{N}(t_{m}-s)dW^{N}(s) \right)\\ 
&=& S_{h}(t_{m})
P_{h}X_{0}+\underset{k=0}{\sum^{m-1}}\left(\int_{t_{k}}^{t_{k+1}}
  S_{h}(t_{m}-s)P_{h}F(X_{k}^{h})ds\right)+ P_{h}P_{N}O(t_{m})\\ 
&=& Z_{m}^{h} +P_{h}P_{N}O(t_{m}),
\end{eqnarray*}
where 
\begin{eqnarray*}
 Z_{m}^{h}&=& S_{h}(t_{m})
 P_{h}X_{0}+\underset{k=0}{\sum^{m-1}}\left(\int_{t_{k}}^{t_{k+1}}
   S_{h}(t_{m}-s)P_{h}F(X_{k}^{h})ds\right)\\ 
 &=& S_{h}(t_{m})
 P_{h}X_{0}+\underset{k=0}{\sum^{m-1}}\left(\int_{t_{k}}^{t_{k+1}}
   S_{h}(t_{m}-s)P_{h}F(Z_{k}^{h} +P_{h}P_{N}O(t_{k}))ds\right).
\end{eqnarray*}
We now estimate 
$\left(\textbf{E}\Vert X(t_{m})-X_{m}^{h}\Vert^{2}\right)^{1/2}$.
We obviously have
\begin{eqnarray}
X(t_{m})-X_{m}^{h}
  &=& \overline{X}_{t_{m}} + O(t_{m})-X_{m}^{h}\nonumber \\
  &=&\overline{X}(t_{m}) + O(t_{m})-\left(Z_{m}^{h}+P_{h}P_{N}
    O(t_{m})\right)\nonumber \\  
  &=& \left(\overline{X}(t_{m}) -Z_{m}^{h}\right)+
  \left(P_{N}(O(t_{m}))-P_{h}P_{N}(O(t_{m}))\right)+
  \left(O(t_{m})-P_{N}(O(t_{m}))\right)\nonumber \\
  &=& I +II +III. \label{eq:IIIIII}
\end{eqnarray}
Then
$\left(\textbf{E}\Vert X(t_{m})-X_{m}^{h}\Vert^{2}\right)^{1/2}
\leq \left(\textbf{E}\Vert I\Vert^{2}\right)^{1/2}
+\left(\textbf{E}\Vert II\Vert^{2}\right)^{1/2} +\left(\textbf{E}\Vert
  III\Vert^{2}\right)^{1/2} $
and we estimate each term. Since the first term will require the most work
we first estimate the other two.

Let us estimate $\left(\textbf{E}\Vert II\Vert^{2}\right)^{1/2}$.
Using the property (\ref{eq:feError}) of the projection $P_{h}$, the equivalence $\Vert. \Vert_{H^{r}(\Omega)}\equiv \Vert
(-A)^{r/2} .\Vert $\;in $\;\mathcal{D}((-A)^{r/2}) $, the Ito isometry
and the fact that the semigroup is a bounded operator yields 
\begin{eqnarray*}
  \textbf{E}\Vert II\Vert^{2} &\leq& C h^{2 r}  \textbf{E} \Vert( -A
  )^{r/2}\int_{0}^{t_{m}} S(t_{m}-s)dW(s)\Vert^{2}\\ 
  &\leq& C h^{2 r}  \int_{0}^{t_{m}} \Vert
  (-A)^{r/2}S(t_{m}-s)\Vert_{L_{2}^{0}}^{2} ds\\ 
  &\leq& C h^{2 r} \int_{0}^{T} \Vert (-A)^{r/2} Q^{1/2}\Vert_{HS}^{2}
  ds. 
\end{eqnarray*}
Thus, since the noise is in $H^{r}$ we have   $(\textbf{E}\Vert
II\Vert^{2})^{1/2}\leq C h^{r}$. 

For the third term $III$
$$\textbf{E}\Vert III\Vert^{2} =  
\textbf{E} \Vert (\text{I}-P_{N})O(t_{m}) \Vert^{2} = 
\textbf{E} \Vert (\text{I}-P_{N})(-A)^{-r/2} (-A)^{r/2}O(t_{m})\Vert^{2},
$$
and so
$$
\textbf{E}\Vert III\Vert^{2}  \leq  
\Vert (\text{I}-P_{N})(-A)^{-r/2} \Vert^2  \textbf{E} \Vert
(-A)^{r/2} O(t_{m})\Vert^2 \leq  C \left( \underset { j \in
    \mathbb{N}^{d} \backslash \mathcal{I}_{N}} {\inf} \lambda_{j}\right)^{-r}.
$$
We now turn our attention to the first term $\textbf{E}\Vert I\Vert^{2}$. 
Using the definition of $T_h$ from \eqref{form1} the first
term $I$ can be expanded 
\begin{eqnarray}
I &=&  T_{h}X_{0} + \underset{k=0}{\sum^{m-1}} \int_{t_{k}
  t}^{t_{k+1}} S( t_{m}-s) F(X(s))-S_{h}(t_{m}-s)
P_{h}F(Z_{k}^{h}+P_{h}P_{N}O(t_{k}))ds \nonumber \\  
 &=& T_{h}X_{0} + \underset{k=0}{\sum^{m-1}}
 \int_{t_{k}}^{t_{k+1}}S_{h}(t_{m}-s) P_{h}(F(X(t_{k})) 
-F(Z_{k}^{h}+P_{h}P_{N}O(t_{k}))))ds\nonumber \\ 
&& +\underset{k=0}{\sum^{m-1}} \int_{t_{k}}^{t_{k+1}}
S_{h}(t_{m}-s)P_{h}(F(X(s))-F (X(t_{k})))ds\nonumber \\ 
&& +\underset{k=0}{\sum^{m-1}}
\int_{t_{k}}^{t_{k+1}}(S(t_{m}-s)-S_{h}(t_{m}-s)P_{h})F(X(s))ds\nonumber\\  
 & =& I_{1}+I_{2}+I_{3}+I_{4}.\nonumber\\  
\label{eq:I1toI5}
\end{eqnarray}
Then
\begin{eqnarray*}
\left(\textbf{E} \Vert I\Vert^{2}\right)^{1/2}\leq \left(\textbf{E}
  \Vert I_{1} \Vert^{2}\right)^{1/2}+\left(\textbf{E} \Vert I_{2}
  \Vert^{2}\right)^{1/2}+\left(\textbf{E} \Vert I_{3}
  \Vert^{2}\right)^{1/2}+\left(\textbf{E} \Vert I_{4}
  \Vert^{2}\right)^{1/2}. 
\end{eqnarray*}
For $I_1$, if $X_{0} \in \mathcal{D}((-A)^{\gamma})\subset H $, equation
(\ref{form2}) of Lemma \ref{lemme1}  gives
\begin{eqnarray*}
 (\textbf{E} \Vert I_{1} \Vert^{2})^{1/2}\leq C 
 t_{m}^{-1/2}h^{2} \left(\textbf{E} \Vert
   X_{0}\Vert^{2}\right)^{1/2} 
\end{eqnarray*}
and if $X_{0} \in  \mathcal{D}(-A) = \HH \cap H^{2}(\Omega)$, equation
 (\ref{form3}) of Lemma \ref{lemme1}  gives
\begin{eqnarray*}
 (\textbf{E} \Vert I_{1} \Vert^{2})^{1/2}\leq C h^{2} \left(\textbf{E}
   \Vert X_{0}\Vert^{2}_{H^{2}(\Omega)}\right)^{1/2}. 
\end{eqnarray*}
If $F$ satisfies \assref{assumption4} (a), then using the Lipschitz
condition, triangle inequality as well as that $S_{h}(t)$ and $P_{h}$
are bounded operators, we have  
\begin{eqnarray*}
  \left(\textbf{E} \Vert I_{2} \Vert^{2}\right)^{1/2} &\leq & C
  \underset{k=0}{\sum^{m-1}}\int_{t_{k}}^{t_{k+1}}\left 
    (\textbf{E} \Vert
    F(X(t_{k}))-F((Z_{k}^{h}+P_{h}P_{N}O(t_{k}))\Vert^{2}\right)^{1/2}ds\\  
  &\leq & C \underset{k=0}{\sum^{m-1}}\int_{t_{k}}^{t_{k+1} 
  }\left(\textbf{E}\Vert X(t_{k}
    )-X_{k}^{h}\Vert^{2}\right)^{1/2}ds.
\end{eqnarray*}
As $I_3$ needs more work let us estimate $I_4$ first. 
Using the fact $P_{h}, S, S_{h}$ are bounded with 
\eqref{form2} of \lemref{lemme1} yields 
\begin{eqnarray*}
  \left(\textbf{E} \Vert I_{4} \Vert^{2}\right)^{1/2} &\leq&
  \underset{k=0}{\sum^{m-1}}\int_{t_{k}}^{t_{k+1}}  \left(\textbf{E}
    \Vert T_{h}(t_{m}-s) F(X(s))\Vert^2 \right)^{1/2} ds \\
  & \leq &  C h^{2} \underset{0\leq s\leq T}{\sup}\left( \textbf E
    \Vert F( X(s)) \Vert^{2}\right)^{1/2} 
  \left(\int_{0}^{t_{m}}\left(t_{m}-s \right)^{-1/2} \right)\\
& \leq &  C h^{2}.\\
\end{eqnarray*}
Let us estimate $\left(\textbf{E} \Vert I_{3}
  \Vert^{2}\right)^{1/2}$. 
We add in and subtract out $O(s)$ and $O(t_{k})$ yields
\begin{eqnarray*}
\lefteqn{ \left(\textbf{E} \Vert I_{3} \Vert^{2}\right)^{1/2} }\\
&=&
 \left(\textbf{E}\Vert\underset{k=0}{\sum^{m-1}}\int_{t_{k} 
   }^{t_{k+1}}S_{h}(t_{m}-s) P_{h}\left( F(X(s))-F(X(t_{k}
     )+O(s)-O(t_{k})))\right) ds\Vert^{2}\right)^{1/2}\\ 
&& + \left(\textbf{E}\Vert\underset{k=0}{\sum^{m-1}}\int_{t_{k}
    }^{t_{k+1}} S_{h}(t_{m}-s)P_{h}\left(F(X(t_{k}
      )+O(s)-O(t_{k}))- F(X(t_{k})))\right)
  ds\Vert^{2}\right)^{1/2}\\ 
&:=& \left(\textbf{E}\Vert
  I_{3}^{1}\Vert^{2}\right)^{1/2}+\textbf{E}\left(\Vert
  I_{3}^{2}\Vert^{2}\right)^{1/2}. 
\end{eqnarray*}
Applying the Lipschitz  condition in \assref{assumption4}(a), using the 
fact the semigroup is bounded and according  to \lemref{lemme2}, for
$X_{0}\in \mathcal{D}((-A)^{\gamma}),\;\;\; 0 \leq \gamma \leq 1$ 
we therefore have  
\begin{eqnarray*}
 \left(\textbf{E}\Vert I_{3}^{1}\Vert^{2}\right)^{1/2}&\leq& 
C
\underset{k=0}{\sum^{m-1}}\int_{t_{k}}^{t_{k+1}}\left(\textbf{E}\Vert
  ( X(s)-O(s))-(X(t_{k})-O(t_{k}))\Vert^{2}\right)^{1/2}\\ 
&\leq& C
\underset{k=0}{\sum^{m-1}}\int_{t_{k}}^{t_{k+1}}(s-t_{k})^{\gamma}ds
\quad \leq \quad C \Delta t^{\gamma}. 
\end{eqnarray*}
Let us now estimate $\textbf{E}\left(\Vert
  I_{3}^{2}\Vert^{2}\right)^{1/2}$.  
The analysis below follows the same steps as in \cite{Jentzen4}
although the approximating semigroup $S_{h}$ is different
here. Applying a Taylor expansion to $F$ gives 
\begin{eqnarray*}
 \textbf{E}\left(\Vert I_{3}^{2}\Vert^{2}\right)^{1/2}\leq
 I_{3}^{21}+I_{3}^{22}+I_{3}^{23}, 
\end{eqnarray*}
with 
\begin{eqnarray*}
I_{3}^{21}&=&\left(\textbf{E}\Vert\underset{k=0}{\sum^{m-1}}\int_{t_{k}}^{t_{k+1}}
  S_{h}(t_{m}-s)P_{h}F'(X(t_{k}))(O(s)-S(s-t_{k})O(t_{k}))ds
  \Vert^{2}\right)^{1/2}\\  
I_{3}^{22}&=&\left(\textbf{E}\Vert\underset{k=0}{\sum^{m-1}}\int_{t_{k}}^{t_{k+1}}
  S_{h}(t_{m}-s)P_{h}F'(X(t_{k}))(S(s-t_{k})O(t_{k})-O(t_{k}))ds
  \Vert^{2}\right)^{1/2}\\ 
I_{3}^{23}&=&
\left(\textbf{E}\Vert\underset{k=0}{\sum^{m-1}}\int_{t_{k}}^{t_{k+1}} 
 S_{h}(t_{m}-s) \int_{0}^{1}
R
(1-r) dr ds\Vert^{2} \right)^{1/2}, \\
R & := &   P_{h}F''(X(t_{k}))+r(O(s)-O(t_{k}))( O(s)-O(t_{k}),O(s)-O(t_{k})).
\end{eqnarray*}
Using the fact that $O(t_{2})-S(t_{2}-t_{1})O(t_{1})$, $0\leq
t_{1}<t_{2} \leq T$ is independent of $\mathcal{F}_{t_{1}}$, one can
show, as in \cite{Jentzen4}, that 
\begin{eqnarray*}
\left(I_{3}^{21}\right)^{2}&=&\underset{k=0}{\sum^{m-1}}\textbf{E}\Vert
\int_{t_{k}}^{t_{k+1}} S_{h}(t_{m}-s) P_{h}
F'(X(t_{k}))(O(s)-S(s-t_{k})O(t_{k}))ds\Vert^{2}.
\end{eqnarray*}
Therefore  as $S_{h}$ is bounded we have 
\begin{eqnarray*}
\lefteqn{ I_{3}^{21}}\\ 
&\leq& 
\left(\underset{k=0}{\sum^{m-1}}\left(\int_{t_{k}}^{t_{k+1}
      } \left(\textbf{E}\Vert S_{h}(t_{m}-s) P_{h}F'(X(t_{k}
      ))(O(s)-S(s-t_{k})O(t_{k}
      ))\Vert^{2}\right)^{1/2} ds\right)^{2}\right)^{1/2}\\ 
&\leq& C \left(\underset{k=0}{\sum^{m-1}}\left(\int_{t_{k}
      }^{t_{k+1}} \left(\textbf{E}\Vert P_{h}F'(X(t_{k}
      ))(O(s)-S(s-t_{k})O(t_{k}))\Vert^{2}\right)^{1/2}
    ds\right)^{2}\right)^{1/2} .
\end{eqnarray*}
Using Fubini's theorem in integration  with \assref{assumption3},
\assref{assumption4}(a) and \propref{prop1} yields 
\begin{eqnarray*}
\lefteqn{I_{3}^{21} }\\
 &\leq& C \Delta
t^{1/2}\left(\underset{k=0}{\sum^{m-1}}\int_{t_{k}}^{t_{k+1}}
  \textbf{E}\Vert P_{h}F'(X(t_{k}))(O(s)-S(s-t_{k})O(t_{k}))\Vert^{2}
  ds\right)^{1/2}\\ 
&\leq& C\Delta
t^{1/2}\left(\underset{k=0}{\sum^{m-1}}\int_{t_{k}}^{t_{k+1}}
  \textbf{E}\Vert (O(s)-S(s-t_{k})O(t_{k}))\Vert^{2} ds\right)^{1/2}\\ 
&\leq& C\Delta
t^{1/2}\left(\underset{k=0}{\sum^{m-1}}\int_{t_{k}}^{t_{k+1}}
  \left(\left(\textbf{E}\Vert O(s)-O(t_{k})\Vert^{2}\right)^{1/2} 
    +\left(\textbf{E}\Vert
  (S(s-t_{k})-\textbf{I})O(t_{k}))\Vert^{2}\right)^{1/2}\right)^{2} ds\right)^{1/2}\\  
&\leq& C\Delta t^{1/2}\left(\underset{k=0}{\sum^{m-1}}\int_{t_{k}}^{t_{k+1}}
  \left((s-t_{k})^{\theta}+(s-t_{k})^{r/2}\left(\textbf{E}\Vert
      O(t_{k})\Vert_{r}^{2} \right)^{1/2}\right)^{2}
  ds\right)^{1/2}\\ 
&\leq& C\Delta t^{1/2+\theta}. 
\end{eqnarray*}
Let us estimate  $I_{3}^{22}$.
\begin{eqnarray*}
\lefteqn{ I_{3}^{22}}\\
&\leq&\underset{k=0}{\sum^{m-1}}\int_{t_{k}}^{t_{k+1}}\left(\textbf{E}\Vert
  S_{h}(t_{m}-s)P_{h}(-A)^{1/2}(-A)^{-1/2}F'(X(t_{k}))(S(s-t_{k})-
  \textbf{I})O(t_{k}))\Vert^{2}\right)^{1/2}ds \\ 
&\leq&\underset{k=0}{\sum^{m-1}}\int_{t_{k}}^{t_{k+1}} \Vert
S_{h}(t_{m}-s)P_{h}(-A)^{1/2}\Vert \left(\textbf{E}\Vert
  (-A)^{-1/2}F'(X(t_{k}))(S(s-t_{k})-
  \textbf{I})O(t_{k}))\Vert^{2}\right)^{1/2}ds.\\ 
\end{eqnarray*}
Since $ P_{h}(-A)^{1/2}=(-A_{h})^{1/2}$
and $S_{h}$ satisfies the smoothing properties analogous  to
$S(t)$  independently of $h$ (see for example \cite{Stig}), and in particular
\begin{eqnarray*}
 \Vert S_{h}(t_{m})(-A_{h})^{1/2} \Vert=\Vert (-A_{h})^{1/2}
 S_{h}(t_{m})\Vert \leq  C t_{m}^{-1/2},\quad t_{m}=m
 \Delta t >0, 
\end{eqnarray*}
we therefore have 
\begin{eqnarray*}
\lefteqn{I_{3}^{22}}\\ 
&\leq& C \underset{k=0}{\sum^{m-1}}\int_{t_{k}}^{t_{k+1}}
\left( t_{m} - s\right)^{-1/2} \left(\textbf{E}\Vert
  (-A)^{-1/2}F'(X(t_{k}))((S(s-t_{k})-
  \textbf{I})O(t_{k}))\Vert^{2}\right)^{1/2}ds.
\end{eqnarray*}
The usual identification of $H=L^{2}(\Omega)$ to its dual yields
\begin{eqnarray*}
\lefteqn{I_{3}^{22}}\\
& \leq & C \underset{k=0}{\sum^{m-1}}\int_{t_{k}}^{t_{k+1}}
\left( t_m - s \right)^{-1/2} 
\left[\textbf{E}
  \left(\underset{\Vert v \Vert \leq 1}{\sup} \vert \langle v,
    (-A)^{-1/2}F'(X(t_{k}))((S(s-t_{k})- \textbf{I})O(t_{k}))\rangle
    \vert \right)^{2}\right]^{1/2} ds 
\end{eqnarray*}
where $\langle,\rangle =(,)$ and we change the notation merely to
emphasize that $H$ is identified to its dual space. 
The fact that $(-A)^{-1/2}$ is self-adjoint implies that 
$\left((-A)^{-1/2}F'(X)\right)^{*}=F'(X)^{*} (-A)^{-1/2}$.
This combined with the fact  that  $H \subset \mathcal{D}((-A)^{1/2})$ thus
$\mathcal{D}((-A)^{-1/2}) \subset
\left(\mathcal{D}((-A)^{1/2})\right)^{*} \subset H^{*}= H $
continuously and \assref{assumption4}(a) yields 
\begin{eqnarray*}
I_{3}^{22} &\leq& C \underset{k=0}{\sum^{m-1}}\int_{t_{k}}^{t_{k+1}}
\left(t_m- t_{k+1}\right)^{-1/2} \\ && .\left(\textbf{E}
  \left(\underset{\Vert v \Vert \leq 1}{\sup} \vert \langle
    F'(X(t_{k}))^{*}(-A)^{-1/2} v,(S(s-t_{k})-
    \textbf{I})O(t_{k})\rangle \vert\right)^{2}\right)^{1/2} ds\\ 
&\leq& C \underset{k=0}{\sum^{m-1}}\int_{t_{k}}^{t_{k+1}} \left(t_m-
  t_{k+1} \right)^{-1/2}\\ && .\left(\textbf{E} \left(\underset{\Vert v
      \Vert \leq 1}{\sup} \Vert  F'(X(t_{k}))^{*}(-A)^{-1/2}
    v\Vert_{1} \Vert (S(s-t_{k})- \textbf{I})O(t_{k}))
    \Vert_{-1}\right)^{2}\right)^{1/2} ds\\ 
&\leq& C \underset{k=0}{\sum^{m-1}}\int_{t_{k}}^{t_{k+1}}
\left(t_m-t_{k+1}\right)^{-1/2}\\ &&.\left(\textbf{E} \left( 1+\Vert 
    X(t_{k})\Vert_{1}\right)^{2}\|\left(S(s-t_{k})-\textbf{I})O(t_{k})\Vert_{-1}\right)^{2}\right)^{1/2}ds\\   
&\leq& C \underset{k=0}{\sum^{m-1}}\int_{t_{k}}^{t_{k+1}} 
\left(t_m-t_{k+1}\right)^{-1/2}\\ &&.\left(\textbf{E} \left( 1+\Vert
    X(t_{k})\Vert_{1}\right)^{4}\right)^{1/4}\left(\textbf{E}\left(\Vert S(s-t_{k})-\textbf{I})O(t_{k})\Vert_{-1}\right)^{4}\right)^{1/4}ds\\ 
&\leq& C \underset{k=0}{\sum^{m-1}}\left( t_m- t_{k+1}
\right)^{-1/2}\\ &&.\left(1+\left(\textbf{E} \Vert 
    X(t_{k})\Vert_{1}^{4}\right)^{1/4}\right)\int_{t_{k}}^{t_{k+1}}
\left(\textbf{E}\left(\|S(s-t_{k})-\textbf{I})O(t_{k})\Vert_{-1}\right)^{4}\right)^{1/4}ds\\ 
&\leq& C \underset{k=0}{\sum^{m-1}}\left(t_m- t_{k+1}\right)^{-1/2}\\&&.\int_{t_{k}}^{t_{k+1}} \Vert
(-A)^{-(r/2+1/2)}\left(S(s-t_{k})-\textbf{I}\right)\Vert
\left(\textbf{E}\Vert O(t_{k})\Vert_{r}^{4}\right)^{1/4}ds\\ 
&\leq& C
\underset{k=0}{\sum^{m-1}}\left(t_m-t_{k+1}\right)^{-1/2}\int_{t_{k}}^{t_{k+1}}
\Vert 
(-A)^{1/2-r/2}(-A)^{-1}(S(s-t_{k})-\textbf{I})\Vert ds.
\end{eqnarray*}
Using \propref{prop1} and the fact that $(-A)^{1/2-r/2}$ is
bounded as $r=1,2$ yields 
\begin{eqnarray*}
I_{3}^{22} &\leq& C \underset{k=0}{\sum^{m-1}}\left(t_m-t_{k+1}\right)^{-1/2}\int_{t_{k}}^{t_{k+1}}\left(s-t_{k}\right)ds\\
&=& C \Delta t^{3/2} \underset{k=0}{\sum^{m-1}}\left(m- k-1\right)^{-1/2}.
\end{eqnarray*}
We can bound the sum above by $2M^{1/2}$,
  therefore we have 
\begin{eqnarray*}
 I_{3}^{21}+ I_{3}^{22} \leq C(\Delta t +\Delta t^{1/2+\theta})\leq
 C(\Delta t^{2\theta}).
\end{eqnarray*}
Let us estimate  $I_{3}^{23}$. Using the fact that 
$S_{h}$ is bounded  and
\assref{assumption4} yields (with $R=
P_{h}F''(X(t_{k})+r(O(s)-O(t_{k})))(O(s)-O(t_{k}),O(s)-O(t_{k}))$)
\begin{eqnarray*}
 I_{3}^{23}&\leq &
 \underset{k=0}{\sum^{m-1}}\int_{t_{k}}^{t_{k+1}}\Vert S_{h}(t_{m}-s)\Vert \int_{0}^{1} \left(\textbf{E}\Vert
R
\Vert^{2} \right)^{1/2}drds\\ 
&\leq & C \underset{k=0}{\sum^{m-1}}\int_{t_{k}}^{t_{k+1}}
\int_{0}^{1} \left (\textbf{E}\Vert
  O(s)-O(t_{k})\Vert_{\mathcal{V}}^{4}\right)^{1/2} dr ds\\ 
&\leq & C \underset{k=0}{\sum^{m-1}}\int_{t_{k}}^{t_{k+1}}
\left(\left (\textbf{E}\Vert
    O(s)-O(t_{k})\Vert_{\mathcal{V}}^{4}\right)^{1/4} \right)^{2} ds\\ 
&\leq & \underset{k=0}{\sum^{m-1}}\int_{t_{k}}^{t_{k+1}} \left(
  s-t_{k}\right)^{2\theta} ds\\ 
&\leq & C (\Delta t)^{2\theta}.
\end{eqnarray*}
Combining $I_{3}^{21}+I_{3}^{22}$ and $I_{3}^{23}$ yields the
following estimate
\begin{eqnarray*}
  \textbf{E}\left(\Vert I_{3}\Vert^{2}\right)^{1/2}\leq  C  (\Delta t^{2\theta})\leq  C  (\Delta t^{\sigma}).
\end{eqnarray*}
Combining the previous estimates for the term $I$ yields for $X_{0}\in
\mathcal{D}(-A)$, 
\begin{eqnarray*}
\left(\textbf{E} \Vert I\Vert^{2}\right)^{1/2}&\leq& C ( h^{2} +\Delta
t^{2 \theta} 
+\underset{k=0}{\sum^{m-1}}\int_{t_{k}}^{t_{k+1}}\left(\textbf{E}\Vert
  X(t_{k})-X_{k}^{h}\Vert^{2}\right)^{1/2})
\end{eqnarray*}
and for $X_{0} \in \mathcal{D}((-A)^{\gamma})$.
\begin{eqnarray*}
\left(\mathbf{E} \Vert I\Vert^{2}\right)^{1/2}&\leq &C
(t_{m}^{-1/2}(h^{2} +\Delta t^{\sigma}) 
+\underset{k=0}{\sum^{m-1}}\int_{t_{k}}^{t_{k+1}}\left(\textbf{E}\Vert
  X(t_{k})-X_{k}^{h}\Vert^{2}\right)^{1/2}).
\end{eqnarray*}
Finally we combine all our estimates on $I$, $II$ and $III$ to get 
$\left(\textbf{E} \Vert I\Vert^{2}\right)^{1/2},\left(\textbf{E} \Vert
  II\Vert^{2}\right)^{1/2}$ and $\left(\textbf{E} \Vert
  III\Vert^{2}\right)^{1/2}$ 
and use the  discrete Gronwall inequality to 
complete the proof.

\subsection{Proof of Theorem \ref{th2} for (\textbf{SETD1}) scheme}
We now prove convergence in $H^{1}(\Omega)$ and 
estimate 
$\left(\mathbf{E}\Vert
  X(t_{m})-X_{m}^{h}\Vert^{2}_{H^1(\Omega)}\right)^{1/2}$.  
For the proof we follow the same steps as in previous section for
\thmref{th1} and we now  estimate \eqref{eq:IIIIII} in the $H^1$ norm.
  
Let  estimate $(\textbf{E}\Vert II\Vert_{H^{1}(\Omega)}^{2})^{1/2}$.
As in the proof of \thmref{th1} in \secref{sec:th1}, 
 using  the regularity of the noise 
$ O(t) \in \mathcal{D}(-A)=H^{2}(\Omega)\cap \HH$, $\forall t\in
[0,T]$ and the property (\ref{eq:feError}) of the projection $P_h$  yields 
\begin{eqnarray*}
  \mathbf{E}\Vert II\Vert_{H^{1}(\Omega)}^{2}&=&\textbf{E} \Vert P_{h}P_{N}(O(t_{m}))-P_{N}(O(t_{m}))\Vert_{H^{1}(\Omega)}^{2}\\
   &\leq& C h^{2}\textbf{E} \Vert (P_{N}(O(t_{m}))\Vert_{H^{2}(\Omega)}\\
    &\leq& C h^{2} \textbf{E} \Vert O(t_{m})\Vert_{H^{2}(\Omega)} \\
&\leq& C h^{2}  \textbf{E} \Vert( -A )\int_{0}^{t_{m}} S(t_{m}-s)dW(s)\Vert^2\;\;\\
&\leq& C h^{2}   \int_{0}^{t_{m}} \Vert (-A)S(t_{m}-s)\Vert_{L_{2}^{0}}^{2} ds\;\;\\
&\leq& C h^{2} \int_{0}^{T} \Vert (-A) Q^{1/2}\Vert_{HS}^{2} ds\;\;\;\\
&\leq& Ch^{2},
\end{eqnarray*}
 thus 
$  (\mathbf{E}\Vert II\Vert_{H^{1}(\Omega)}^{2})^{1/2} \leq C h.$

Using  the regularity of the noise again  and the equivalency  $\Vert.\Vert_{H^{1}(\Omega)}\equiv \Vert(-A)^{1/2}.\Vert $,  we also have 
\begin{eqnarray*}
 \mathbf{E}\Vert III\Vert_{H^{1}(\Omega)}^{2} & =&  \textbf{E} \Vert (\text{I}-P_{N})O(t_{m}) \Vert_{H^{1}(\Omega)}^{2}\\
    & =&  \textbf{E} \Vert (\text{I}-P_{N})(-A)^{-1} (-A)^{1}O(t_{m})\Vert_{H^{1}(\Omega)}^2\\
      & =& \textbf{E} \Vert(-A)^{1/2} (\text{I}-P_{N})(-A)^{-1} (-A)^{1}O(t_{m})\Vert_{H^{1}(\Omega)}^2\\
 &\leq&  \Vert (-A)^{1/2}(\text{I}-P_{N})(-A)^{-1} \Vert^2  \textbf{E} \Vert (-A)O(t_{m}) \Vert^2\\
 &\leq& \Vert (-A)^{1/2}(\text{I}-P_{N})(-A)^{-1}\Vert^2 \textbf{E} \Vert (-A)O(t_{m}) \Vert^2\\
 &\leq&  C \left( \underset { j \in \mathbb{N}^{d}  \backslash \mathcal{I}_{N}} {\inf}  \lambda_{j} \right)^{-1}.
 \end{eqnarray*}
We now estimate the term $I$ from \eqref{eq:IIIIII} in the
$H^1(\Omega)$ norm noting that from \eqref{eq:I1toI5} we have 
$I=I_1+I_2+I_3+I_4$. 
Estimates on $I_1$ follow immediately from  equations (\ref{form4}) and
(\ref{form5}) of \lemref{lemme1}, and then 
for $I_1$, if $X_{0} \in \mathcal{D}((-A)^{(\gamma+1)/2})\subset V $, equation
(\ref{form2}) of Lemma \ref{lemme1}  gives
\begin{eqnarray*}
 (\textbf{E} \Vert I_{1} \Vert_{H^{1}(\Omega)}^{2})^{1/2}\leq C 
 t_{m}^{-1/2}h \left(\textbf{E} \Vert
   X_{0}\Vert_{H^{1}(\Omega)}^{2}\right)^{1/2} 
\end{eqnarray*}
and if $X_{0} \in  \mathcal{D}((-A)) = \HH \cap H^{2}(\Omega)$,
\begin{eqnarray*}
 (\textbf{E} \Vert I_{1} \Vert_{H^{1}(\Omega)}^{2})^{1/2}\leq C h \left(\textbf{E} \Vert X_{0}\Vert^{2}_{H^{2}(\Omega)}\right)^{1/2}.
\end{eqnarray*}
If $F$ satisfies \assref{assumption4} (b), then  
 using the Lipschitz  condition, the triangle inequality, the fact
 that $P_{h}$ is an bounded operator  and $S_{h}$ satisfies
 the smoothing property analogous to $S(t)$  independently of $h$
 \cite{Stig}, i.e.   
\begin{eqnarray*}
 \Vert S_{h}(t)v\Vert_{H^{1}(\Omega)}^{2}\leq C t^{-1/2} \Vert v\Vert
 \qquad v \in V_{h}\quad t>0,
\end{eqnarray*}
we  have
\begin{eqnarray*}
\lefteqn{(\mathbf{E} \Vert I_{2} \Vert_{H^{1}(\Omega)}^{2})^{1/2}}\\ &\leq &
\underset{k=0}{\sum^{m-1}}\int_{t_{k}}^{t_{k+1}} (\textbf{E} \Vert
S_{h}(t_{m}-s)P_{h}(F(X(t_{k}))-F((Z_{k}^{h}+P_{h}P_{N}O(t_{k}))))\Vert_{H^{1}(\Omega)}^{2})^{1/2}ds\\ 
&\leq &  C \underset{k=0}{\sum^{m-1}}\int_{t_{k}}^{t_{k+1}} (t_{m}-s)^{-1/2}(\mathbf{E} \Vert F(X(t_{k}))-F((Z_{k}^{h}+P_{h}P_{N}O(t_{k}))\Vert^{2})^{1/2})ds\\
 &\leq & C \underset{k=0}{\sum^{m-1}}\int_{t_{k}}^{t_{k+1}}(t_{m}-s)^{-1/2}(\mathbf{E}\Vert X(t_{k})-X_{k}^{h}\Vert_{H^{1}(\Omega)}^{2})^{1/2}ds.
\end{eqnarray*}

Once again using  Lipschitz  condition, triangle inequality, smoothing
property of $ S_{h}$, but with Lemma \ref{lemme2} gives 
\begin{eqnarray*}
\lefteqn{(\textbf{E} \Vert I_{3} \Vert_{H^{1}(\Omega)}^{2})^{1/2}}\\
 &\leq &\underset{k=0}{\sum^{m-1}}\int_{t_{k}}^{t_{k+1}} (\textbf{E}  \Vert
 S_{h}(t_{m}-s)P_{h}(F(X(s))-F(X(t_{k}))\Vert_{H^{1}(\Omega)}^{2})^{1/2} )ds\\
&\leq & C \underset{k=0}{\sum^{m-1}}\int_{t_{k}}^{t_{k+1}}
(t_{m}-s)^{-1/2}(\textbf{E}  \Vert F(X(s))-F(X(t_{k}))\Vert)^{1/2}ds\\ 
&\leq & C\underset{k=0}{\sum^{m-1}}\int_{t_{k}}^{t_{k+1}}
(t_{m}-s)^{-1/2}(\textbf{E}  \Vert
X(s)-X(t_{k})\Vert_{H^{1}(\Omega)}^{2})^{1/2}ds\\ 
&\leq&  C \left (\underset{k=0}{\sum^{m-1}} \int_{t_{k}}^{t_{k+1}}
  (t_{m}-s)^{-1/2}(s- t_{k})^{\gamma/2}ds\right)\\ && .\left(\textbf{E}
  \Vert X_{0}\Vert_{\gamma+1}^{2}+\left(\textbf{E} \underset{0\leq
      s\leq T}{\sup} \Vert F(X(s))\Vert_{H^{1}(\Omega)}
  \right)^{2}+1\right)^{1/2}\\ 
&\leq&  C \left( \Delta t^{\gamma/2}
  \underset{k=0}{\sum^{m-1}}\int_{t_{k}}^{t_{k+1}}(t_{m}-s)^{-1/2}\right)\\ && 
.\left(\mathbf{E} \Vert X_{0}\Vert_{\gamma+1}^{2}+\left(\textbf{E} \underset{0\leq s\leq T}{\sup} \Vert F(X(s))\Vert_{H^{1}(\Omega)} \right)^{2}+1\right)^{1/2}.  
\end{eqnarray*}
As in the previous theorem, we use the fact that
$$\underset{k=0}{\sum^{m-1}}\int_{t_{k}}^{t_{k+1}}(t_{m}-s)^{-1/2}\leq 2
\sqrt{ T}.$$

Then if  $X_{0} \in \mathcal{D}((-A)^{(\gamma+1)/2})$ we have finally found
\begin{eqnarray*}
  (\mathbf{E} \Vert I_{3} \Vert_{H^{1}(\Omega)}^{2})^{1/2} \leq C
  (\Delta t)^{\gamma/2}\left(\textbf{E} \Vert
    X_{0}\Vert_{\gamma+1}^{2}+\left(\mathbf{E} \underset{0\leq s\leq
        T}{\sup} \Vert F(X(s))\Vert_{H^{1}(\Omega)}
    \right)^{2}+1\right)^{1/2}. 
\end{eqnarray*}
In  the same way,  if  $X_{0} \in \mathcal{D}(-A)$  we obviously have 
$(\mathbf{E} \Vert I_{3} \Vert_{H^{1}(\Omega)}^{2})^{1/2} \leq C (\Delta t)^{1/2-\epsilon}$
by taking $\gamma = 1-\epsilon$ in Lemma  \ref{lemme2},\;$\epsilon>0$
small enough. 

If $F(X)\in V$ then by \eqref{form4} of \lemref{lemme1} we find 
\begin{eqnarray*}
 (\mathbf{E} \Vert I_{4} \Vert_{H^{1}(\Omega)}^{2})^{1/2} &\leq& \underset{k=0}{\sum^{m-1}}\int_{t_{k}}^{t_{k+1}}  (\mathbf{E} \Vert T_{h}(t_{m}-s) F(X(s))\Vert_{H^{1}(\Omega)}^2)^{1/2} ds\\
 & \leq & C h \left(\underset{k=0}{\sum^{m-1}}\int_{t_{k}}^{t_{k+1}}  (t_{m}-s)^{-1/2}ds \right) \left(\underset{0\leq s\leq T}{\sup} \textbf E \Vert F( X(s)) \Vert_{H^{1}(\Omega)}^{2}\right)^{1/2}\\
& \leq & C h.
\end{eqnarray*}
Combining our estimates, for  $F(X)\in V$ we have  that : if $
X_{0}\in \mathcal{D}((-A)^{(\gamma+1)/2})$ then 
\begin{eqnarray*}
(\mathbf{E} \Vert I\Vert_{H^{1}(\Omega)}^{2})^{1/2}&\leq& C
(t_{m}^{-1/2}\,h +\Delta t^{\gamma/2})
\\ &+&\underset{k=0}{\sum^{m-1}}\int_{t_{k}}^{t_{k+1}}(t_{m}-s)^{-1/2}(
\mathbf{E}\Vert X(t_{k})-X_{k}^{h}\Vert_{H^{1}(\Omega)}^{2})^{1/2})ds. 
\end{eqnarray*}
If $ X_{0}\in \mathcal{D}(-A)$ then
 \begin{eqnarray*}
(\mathbf{E} \Vert I\Vert_{H^{1}(\Omega)}^{2})^{1/2}&\leq& C (h
+\Delta t^{(\frac{1}{2}-\epsilon)} )
\\&+&\underset{k=0}{\sum^{m-1}}\int_{t_{k}}^{t_{k+1}}(t_{m}-s)^{-1/2}
(\textbf{E}\Vert X(t_{k})-X_{k}^{h}\Vert_{H^{1}(\Omega)}^{2})^{1/2})ds.   
\end{eqnarray*}
where $C>0$ depending of the $T$, the initial solution $X_{0}$, the mild solution $X$, the nonlinear function $F$.\\

Combining our estimates $\left(\textbf{E} \Vert
  I\Vert_{H^{1}(\Omega)}^{2}\right)^{1/2},\left(\textbf{E} \Vert
  II\Vert_{H^{1}(\Omega)}^{2}\right)^{1/2}$ and $\left(\textbf{E} \Vert
  III\Vert_{H^{1}(\Omega)}^{2}\right)^{1/2}$ and using the discrete Gronwall lemma 
concludes the proof. 

\subsection{Proofs for the (\textbf{SETD0}) scheme}
Recall that
\begin{eqnarray*}
 Y_{m}^{h}&=& e^{\Delta t A_{h}}\left( Y_{m-1}^{h}+\Delta t P_{h}F(Y_{m-1}^{h})\right)+  P_{h}\int_{t_{m-1}}^{t_{m}} e^{(t_{m}-s)A_{N}}d W^{N}(s)\\
 &=& S_{h}(t_{m})
 P_{h}X_{0}+\underset{k=0}{\sum^{m-1}}\left(\int_{t_{k}}^{t_{k+1}}
   S_{h}(t_{m}-t_{k})P_{h}F(Y_{k}^{h})ds \right.\\ && + \left. P_{h}\int_{t_{k}}^{t_{k+1}}
   S_{N}(t_{m}-s)dW^{N}(s) \right)\\ 
&=& S_{h}(t_{m})
P_{h}X_{0}+ \underset{k=0}{\sum^{m-1}}\left(\int_{t_{k}}^{t_{k+1}}
  S_{h}(t_{m}-t_{k})P_{h}F(Y_{k}^{h})ds\right)+ P_{h}P_{N}O(t_{m})\\ 
&=& Z_{m}^{h} +P_{h}P_{N}O(t_{m}).
\end{eqnarray*}
As in the Theorem \ref{th1} 
we obviously have
\begin{eqnarray}
  \lefteqn{X(t_{m})-Y_{m}^{h}}\nonumber\\
  &=& \overline{X}_{t_{m}} + O(t_{m})-Y_{m}^{h}\nonumber \\
  &=&\overline{X}(t_{m}) + O(t_{m})-\left(Z_{m}^{h}+P_{h}P_{N} O(t_{m})\right)\nonumber \\ 
  &=& \left(\overline{X}(t_{m}) -Z_{m}^{h}\right)+
  \left(P_{N}(O(t_{m}))-P_{h}P_{N}(O(t_{m}))\right)+
  \left(O(t_{m})-P_{N}(O(t_{m}))\right)\nonumber \\
  &=& I +II +III\nonumber. \\ 
\end{eqnarray}
The proofs are therefore as \cite[Theorem 2.6 and Theorem 2.7]{GTambue} 
but with $S_{h,\Delta t}^{m-k}$ replaced by  $S_{h}(t_{m}-t_{k})$ and
using the similar estimates as in the proofs  of Theorem \ref{th1} and
Theorem \ref{th2} for the \SETD scheme. 

\section{Implementation \& numerical results}
\label{sec:sim}
\subsection{Efficient computation of the action of  $\varphi_{i},\; i=0,1$}
The key element in the stochastic exponential schemes is 
computing the matrix exponential functions, the so called  $\varphi_{i}-$
functions.  
It is well known that a standard \Pade approximation for a matrix 
exponential is not an efficient method for large scale 
problems~\cite{SID,HSW,CMCVL}. 

Here we focus on the real fast \Leja points and the Krylov subspace
 techniques to evaluate the action of the exponential matrix function
 $\varphi_{i} (\Dt\, A_{h})$ on a vector $\underbar{v}$, 
instead of computing  the full exponential function 
$\varphi_{i} (\Dt\,A_{h})$ 
 as in a standard \Pade approximation. The details of the real fast \Leja points technique and~\cite{kry,SID,Antoine} for the Krylov subspace technique 
are given in~\cite{LE2,LE1,LE}. We give a brief summary below.
In \cite{Antoine} we have compared the efficiency of the two techniques for deterministic advection-- diffusion--reaction.

\subsubsection{Krylov space subspace technique}

The main idea of the Krylov subspace technique is to  approximate the
action of the exponential matrix function  $\varphi_{i}(\Dt A_{h})$  on a
vector $\underbar{v}$ by projection onto a small Krylov subspace
$K_{m} =\text{span} \left\lbrace
  \underbar{v},A_{h} \underbar{v},\ldots,  A_{h}^{m-1}
  \underbar{v}\right\rbrace $~\cite{SID}. The approximation is formed
using an orthonormal  basis of $\mathbf{V}_{m} =
\left[\underbar{v}_{1} , \underbar{v}_{2},\ldots, \underbar{v}_{m}
\right] $ of the Krylov subspace $K_{m}$ and of its completion
$\mathbf{V}_{m+1}=\left[\mathbf{V}_{m},\underbar{v}_{m+1}\right]$. The
basis is found by Arnoldi iteration 
which uses stabilised Gram-Schmidt to produce a
sequence of vectors that span the Krylov subspace.\\ 
Let $\underbar{e}_{i}^{j}$ be the $i^{\textrm{th}}$ standard basis vector of $\mathbb{R}^{j}$.
 We approximate $\varphi_{i}(\Dt \,A_{h})\underbar{v}$ by
\begin{eqnarray}
  \varphi_{i}(\Dt \,A_{h})\underbar{v}
  &\approx & \Vert\underbar{v} \Vert_{2}\mathbf{V}_{m+1}\varphi_{i} ( \Delta
  t\, \overline{\mathbf{H}}_{m+1})\underbar{e}_{1}^{m+1}
  \label{app}
\end{eqnarray}
with
$$
\overline{\mathbf{H}}_{m+1} = \left( \begin{array}{cc}
    \mathbf{H}_{m}&\b0\\
    0,\cdots,0, h_{m+1,m} & 0
  \end{array}\right) \qquad \text{where} \qquad
\mathbf{H}_{m} = \mathbf{V}_{m}^{T} A_{h}\mathbf{V}_{m}=[h_{i,j}].
$$
The coefficient $h_{m+1,m}$ is recovered in the last iteration of
Arnoldi's iteration \cite{SID,Antoine,kry}.
For a small Krylov subspace (i.e, $m$ is small) a standard \Pade
approximation can be used to compute
$\varphi_i(\Dt\overline{\mathbf{H}}_{m+1})$, but a efficient way 
used in ~\cite{SID} is to recover
$\varphi_i(\Dt\overline{\mathbf{H}}_{m+1})\underbar{e}_{1}^{m+1}$  
directly from the \Pade approximation of the exponential of a matrix
related to $\mathbf{H}_{m}$~\cite{SID}.  
In our implementation we
use the functions  {\tt expv.m} and {\tt phiv.m} of the package
Expokit~\cite{SID}, which used the efficient  technique specified
above. 

\subsubsection{Real fast \Leja points technique}
For a given vector $\underbar{v}$, the real fast \Leja points
approximate $\varphi_{i}(\Dt \,A_{h})\underbar{v} $  
 by $P_{m}(\Dt \,A_{h}) \underbar{v}$, where $P_{m}$ is an
 interpolation polynomial of  degree $m$ of $\varphi_{i}$  
 at the sequence of points $\left\lbrace \xi_{i}
 \right\rbrace_{i=0}^{m} $ called spectral real fast \Leja points.  
These points $\left\lbrace \xi_{i} \right\rbrace_{i=0}^{m}$
 belong to the spectral focal interval  $\left[ \alpha, \beta\right]$
 of the matrix  $\Dt A_{h}$, i.e. the focal interval of the 
smaller ellipse containing all the eigenvalues of  $\Dt
\,A_{h}$. This spectral interval can be estimated by  the well
known  Gershgorin circle theorem~\cite{GC}. In has been shown that as
the degree of the polynomial increases and hence the number of  
\Leja points increases, convergence is achieved~\cite{LE2}, i.e.
\begin{eqnarray}
  \underset {m \rightarrow \infty}{\lim}\Vert\varphi_{i}(\Dt A_{h}
  )\underbar{v}- P_{m}(\Dt \,A_{h})\underbar{v}\Vert_{2}=0, 
\end{eqnarray}
where $\Vert .\Vert_{2}$ is the standard Euclidean norm. For a real
interval $\left[ \alpha, \beta\right]$, a sequence of real fast \Leja
points $\left\lbrace \xi_{i}\right\rbrace_{i=0}^{m}$ 
is defined recursively as follows. Given an initial point $\xi_{0}$,
usually $\xi_{0}=\beta$, the sequence of fast \Leja points is generated by 
\begin{eqnarray}
\label{lejaa}
\underset {k=0}{\prod^{j-1}}\vert \xi_{j}-\xi_{k} \vert =\underset {\xi \in
  \left[\alpha, \beta \right] } {\max} \underset
{k=0}{\prod^{j-1}}\mid \xi-\xi_{k} \mid \quad j=1,2,3,\cdots. 
\end{eqnarray}
We use the Newton's form of the interpolating polynomial $P_{m}$ given
by 
\begin{eqnarray}
\label{approle}
 P_{m}(z)=\varphi_{i}\left[ \xi_{0}\right] + \underset{j=1}{\sum
   ^{m}}\varphi_{i}\left[ \xi_{0},\xi_{1}, \cdots,\xi_{j}\right]
 \underset{k=0}{\prod^{j-1}}\left( z-\xi_{k} \right) 
\end{eqnarray}
where the divided differences $\varphi_1[\bullet]$ are defined recursively by
\begin{eqnarray}
\label{div}
\left\lbrace \begin{array}{l}
    \varphi_{i}\left[ \xi_{j}\right] =\varphi_{i}( \xi_{j})\\
    \newline\\
    \varphi_{i}\left[ \xi_{j},\xi_{j+1},
      \cdots,\xi_{k}\right]:=\dfrac{\varphi_{i}\left[
        \xi_{j+1},\xi_{j+2}, \cdots,\xi_{k}\right]-\varphi_{i}\left[
        \xi_{j},\xi_{j+1}, \cdots,\xi_{k-1}\right]}{\xi_{k}-\xi_{j}}. 
  \end{array}\right.
\end{eqnarray}
An algorithm  to compute the action  the action of the exponential
matrix function $\varphi_{i} (\Dt\,A_{h})$ on a vector $\underbar{v}$
can be found  in \cite{Antoine} where the standard way is used to
computer the divided differences. Due to cancellation errors this
standard way cannot produce accurate divided differences with
magnitude smaller than machine precision. Here we used the efficient
way to computer the divided differences \cite{LE2,McM}.

In \cite{Reichel} it is shown that \Leja points for the interval
$[-2,2]$ assure optimal accuracy, thus  for the spectral focal
interval $\left[ \alpha, \beta\right]$  of the matrix
$\Dt A_{h}$, it is convenient to interpolate, by a change of
variables, the function $\varphi_{i}( c+\gamma \xi) $ of the
independent variable $\xi \in [-2, 2]$   with  $c=(\alpha+\beta)/2$
and $\gamma =(\beta-\alpha)/4$. It can be shown \cite{McM} that the
divided differences of a function $f(c+\gamma \xi) $ of the
independent variable $\xi$  at the  points 
$\left\lbrace \xi_{i}\right\rbrace_{i=0}^{m} \subset [-2, 2] $ are the
first column of the matrix function  $f(\mathbf{L}_{m})$, where
$$
\mathbf{L}_{m} =c \textbf{I}_{m+1}+\gamma\mathbf{\widehat{L}}_{m},\;\;\;
\mathbf{\widehat{L}}_{m}= \left( \begin{array}{cccccc}
    \xi_{0} & & & &&\\
    1& \xi_{1}& & &&\\
    & 1& \ddots&&&\\
    & & \ddots&\ddots&&\\
    & & &\ddots&\ddots\\
    & & &&1& \xi_{m}\\
  \end{array}\right) 
$$
We then conclude that  the divided differences of $\varphi_{i}(
c+\gamma \xi)$ of the independent variable $\xi \in [-2, 2]$ at  the
points $\left\lbrace \xi_{i}\right\rbrace_{i=0}^{m} \subset [-2, 2]$ 
is $\varphi_{i}(\mathbf{L}_{m})e_{1}^{m+1}$ where  $e_{1}^{m+1}$ is
the first standard basis vector of $\mathbb{R}^{m+1}$. Taylor
expansion of order $p$ with scaling and squaring   
is used in \cite{LE2,McM} to compute
$\varphi_{i}(\mathbf{L}_{m})e_{1}^{m+1}$. 
In practice the real fast \Leja points is computed once in the
interval  $[-2, 2]$ and reused at each  time step during the
computation of the divided differences.  
We use the efficient algorithm of Baglama et al.~\cite{LE1} to compute
the real fast \Leja points in $[-2, 2]$. 

\subsubsection{Numerical construction of noise}

We relate the decay of the eigenvalues $q_i$ of $Q$ in \eqref{eq:W}
to the covariance function and discuss
implementation. For concreteness we examine $A$ on
$[0,L_1] \times [0,L_2]$ with Neumann boundary conditions.
For the noise in $H^r$, $r=1,2$ we take the following values for
$\left\lbrace q_{i,j}\right\rbrace_{i+j > 0}$ in the representation 
(\ref{eq:W}) 
\begin{eqnarray}
  \label{noise2}
  q_{i,j}= \Gamma/\left( i+j\right)^{r},\qquad r>0.
\end{eqnarray}
We call noise in $H^{r}$ when the eigenvalues satisfy (\ref{noise2}). 
Consider the covariance operator $Q$ with the following covariance
function (kernel) with strong exponential decay
\cite{shardlow05,GrcaOjlvoSncho}    
\begin{eqnarray*}
  C_{r}((x_{1},y_{1});(x_{2},y_{2}))=\dfrac{\Gamma}{4 b_{1}b_{2}} \exp
  \left(-\dfrac{\pi}{4}\left[\dfrac{\left( x_{2}-x_{1}\right)^{2}
      }{b_{1}^{2}}+ \dfrac{\left( y_{2}-y_{1}\right)^{2}
      }{b_{2}^{2}}\right] \right)
\end{eqnarray*}
where $b_{1},b_{2}$ are  spatial correlation lengths in $x-$ axis and
y- axis respectively and  $\Gamma>0$. This covariance function is frequently
used in geosciences to generated a random permeability (see
\cite{Antoine,WuanT}). 
It is well known that the eigenfunctions
$\{e_{i}^{(1)}e_{j}^{(2)}\}_{i,j\geq 0} $ 
of the operator $A=D \varDelta$ with Neumann boundary conditions are
given by
$$e_{0}^{(l)}=\sqrt{\dfrac{1}{L_{l}}},\quad \lambda_{0}^{(l)}=0,\quad
e_{i}^{(l)}=\sqrt{\dfrac{2}{L_{l}}}\cos(\lambda_{i}^{(l)}x),\quad
\lambda_{i}^{(l)}=\dfrac{i  \pi }{L_{l}}$$ 
with $l \in \left\lbrace 1, 2 \right\rbrace$, $i=1, 2, 3, \cdots$
and corresponding eigenvalues $ \{\lambda_{i,j}\}_{i,j\geq 0} $ given by 
\begin{eqnarray*}
\lambda_{i,j}= (\lambda_{i}^{(1)})^{2}+ (\lambda_{j}^{(2)})^{2}.
\end{eqnarray*}
In order  to put  the noise $W$  in form of the representation
(\ref{eq:W}), let us give the following lemma. 
\begin{lemma}
\label{exp}
 Let  $b$ and $\lambda$ be two real numbers. We have the following statement
\begin{eqnarray*}
 \int_{-\infty}^{+ \infty}\exp \left(-\dfrac{\pi}{4}\left(\frac{ x^{2}
     }{b^{2}}\right) \right) \cos(\lambda x) dx =  2 b\exp\left[
   -\dfrac{1}{\pi}\left(\lambda b\right)^{2} \right]. 
\end{eqnarray*}
\end{lemma}

\begin{proof} Note that
\begin{eqnarray*}
\lefteqn{\int_{-\infty}^{+ \infty}\exp \left(-\dfrac{\pi}{4}\left(\frac{
      x^{2}}{b^{2}}\right) \right) \cos(\lambda x)dx}\\
 &=&
\dfrac{1}{2}\int_{-\infty}^{+ \infty}\left(  e^{ -\left(\dfrac{
        \pi\,x^{2} }{4\,b^{2}}-i\lambda x  \right)}+  e^{ -\left(
      \dfrac{ \pi\,x^{2} }{4\,b^{2}}+i\lambda x  
 \right)}\right)dx \\
&=&\dfrac{1}{2} e^{-\dfrac{\left( \lambda b\right)^{2}
  }{\pi}}\int_{-\infty}^{+ \infty}\left(  e^{ -\left(
      \dfrac{\sqrt{\pi}}{2b}\,x-i \dfrac{\lambda b}{\sqrt{\pi}}
    \right)^{2}}+ e^{ -\left( \dfrac{\sqrt{\pi}}{2b}\,x+i
      \dfrac{\lambda b}{\sqrt{\pi}}  \right)^{2}}\right)dx. \\ 
\end{eqnarray*}
Since for any contour $(C)$ in complex plane  we have
$\oint_{C}\exp(-z^{2})dz=0$, taking $(C)$ to be a rectangle with
vertexes in complex plan $ -a, a, a+id,-a+ id$ yields 
\begin{eqnarray*}
 \oint_{C}\exp(-z^{2})dz=\int_{-a}^{a}\,e^{-x^{2}}dx +i \int_{0}^{d}
 e^{-(a+iy)^{2}}dy-\int_{-a}^{a}e^{-(x+id)^{2}}dx-i
 \int_{0}^{d}e^{-(-a+iy)^{2}}dy.   
\end{eqnarray*}
 Since
\begin{eqnarray*}
 \vert \int_{0}^{d} e^{-(\pm a+iy)^{2}}dy\vert = \vert \int_{0}^{d} e^{(- a^{2}\pm i2a y+y^{2})}dy\vert\leq e^{-a^{2}}\int_{0}^{d} e^{y^{2}}dy \rightarrow 0 \quad \text{when}\quad a \rightarrow + \infty
\end{eqnarray*}
then when $a \rightarrow + \infty$, we have
\begin{eqnarray*}
 \int_{-\infty}^{\infty}\,e^{-x^{2}}dx=\int_{-\infty}^{\infty}e^{-(x+id)^{2}}dx \quad \quad \text{for   all } \quad d\in \mathbb{R}.
\end{eqnarray*}
Using previous results allows us to have finally 
\begin{eqnarray*}
\int_{-\infty}^{+ \infty}\exp \left(-\dfrac{\pi}{4}\left(\dfrac{ x^{2} }{b^{2}}\right) \right) \cos(\lambda x)dx &=&e^{-\dfrac{\left( \lambda b\right)^{2} }{\pi}}\int_{-\infty}^{+ \infty}e^{ -\left( \dfrac{\sqrt{\pi}}{2b}\,x\right)^{2}} dx\\
&=& 2 b e^{-\dfrac{\left( \lambda b\right)^{2} }{\pi}}
\end{eqnarray*}
by using the fact that $\int_{-\infty}^{+ \infty}e^{-x^{2}}dx =\sqrt{\pi}$.\\
\end{proof}
Recall \cite{DaPZ} that the covariance operator $Q$ may be defined for $f \in
L^{2}(\Omega)$ by  
\begin{eqnarray*}
 Qf(x) = \int_{\Omega}C_{r}((x,y) f(y)dy.
\end{eqnarray*}
Assume that the eigenfunctions of the operator $Q$ are the same as the
eigenfunctions of $-A$, for $b_{i}\ll L_{i}$ and using the strong
exponential decay of $C_{r}$ we have : 
\begin{eqnarray*}
\label{eig}
   \lefteqn{4 b_{1} b_{2} \int_{0}^{L_{1}} \int_{0}^{L_{2}}
  C_{r}((x_{1},y_{1});(x_{2},y_{2})) \cos(\lambda_{i}^{(1)} x_{2}
  )\cos(\lambda_{j}^{(2)} y_{2} )dy_{2} dx_{2}}\\ 
&=&\Gamma\;\int_{0}^{L_{1}} \exp
\left(-\dfrac{\pi}{4}\left(\dfrac{\left( x_{2}-x_{1}\right)^{2}
    }{b_{1}^{2}}\right) \right) \cos(\lambda_{i}^{(1)} x_{2} )dx_{2}
\\ &&\times \int_{0}^{L_{2}} \exp \left(-\dfrac{\pi}{4}\left[\dfrac{\left(
        y_{2}-y_{1}\right)^{2} }{b_{2}^{2}}\right] \right)
\cos(\lambda_{j}^{(2)} y_{2} ) dy_{2}\\ 
&=& \Gamma\; \int_{-x_{1}}^{L_{1}-x_{1}}\exp
\left(-\dfrac{\pi}{4}\left(\dfrac{ x^{2} }{b_{1}^{2}}\right) \right)
\cos(\lambda_{i}^{(1)} (x+x_{1} ))dx \\ && \times \int_{-y_{1}}^{L_{2}-y_{1}}
\exp \left(-\dfrac{\pi}{4}\left(\dfrac{ x^{2} }{b_{2}^{2}}\right)
\right) \cos(\lambda_{j}^{(2)}(x+y_{1})dx)\\ 
&\approx& \Gamma \;\int_{-\infty}^{+ \infty}\exp
\left(-\dfrac{\pi}{4}\left(\dfrac{ x^{2} }{b_{1}^{2}}\right) \right)
\cos(\lambda_{i}^{(1)} (x+x_{1} ))dx \\ && \times \int_{-\infty}^{+ \infty}
\exp \left(-\dfrac{\pi}{4}\left(\dfrac{ x^{2} }{b_{2}^{2}}\right)
\right) \cos(\lambda_{j}^{(2)}(x+y_{1}))dx\\ 
&=& 4 b_{1} b_{2} \cos(\lambda_{i}^{(1)}\,x_{1})
\cos(\lambda_{j}^{(2)}\,y_{1})\,\Gamma\,\exp\left(
  -\dfrac{1}{\pi}\left((\lambda_{i}^{(1)}b_{1})^{2}+(\lambda_{j}^{(2)}b_{2})^{2}\right) \right).   
\end{eqnarray*}
It is important to notice that in the previous expressions we have
used the fact that 
\begin{eqnarray*}
  \int_{-\infty}^{+ \infty}\exp \left(-\dfrac{\pi}{4}\left(\dfrac{
        x^{2} }{b_{i}^{2}}\right) \right) \cos(\lambda_{j}^{(i)}x)dx
  &= & 2 b_{i}\exp\left[ -\dfrac{1}{\pi}\left((\lambda_{j}^{(i)}
      b_{i})^{2}\right) \right] \;\;\;\; i \in \left\lbrace
    1,2\right\rbrace  
\end{eqnarray*}
by \lemref{exp} and
\begin{eqnarray*}
\int_{-\infty}^{+ \infty}\exp \left(-\dfrac{\pi}{4}\left(\dfrac{ x^{2} }{b_{i}^{2}}\right) \right) \sin(\lambda_{j}^{(i)}x)dx &=&0 
\end{eqnarray*}
because the integrand is an odd function.
Then the corresponding values of $\left\lbrace q_{i,j}\right\rbrace_{i+j > 0}$ in the representation (\ref{eq:W}) is given by
\begin{eqnarray*}
 q_{i,j}= \Gamma \exp\left[ -\dfrac{1}{2 \pi}\left((\lambda_{i}^{(1)}b_{1})^{2}+(\lambda_{j}^{(2)}b_{2})^{2}\right) \right].
\end{eqnarray*}

During our simulation, the process
$$
 \widehat{O}_{k}=\int_{t_{k}}^{t_{k+1}} e^{(t_{k+1}-\tau)A_{N}}d W^{N}(\tau)
$$
is  generated  in Fourier space as in \cite{Jentzen4} by  applying
the Ito isometry in each mode, which yields 
\begin{eqnarray}
 ( e_{i},\widehat{O}_{k}))=  e^{-\lambda_{i} \Delta
   t}\left(\dfrac{q_{i}}{2 \lambda_{i}} \left(1-e^{-2
       \lambda_{i}\Delta t} \right)\right)^{1/2}R_{i,k},
\label{eq:noiseupdate}
\end{eqnarray}
$i\in \mathcal{I}_{N}=\left\lbrace 1,2,3,...,N\right\rbrace^{2}$,
 $k=0,1,2..., M-1$ and $R_{i,k}$ are independent, standard normally
distributed random variables with means $0$ and variance $1$.
For efficient computations we use the inverse fast Fourier transform
or some variant : eg for Neumann boundary conditions we use the
inverse discrete cosine transform. 

The exponential functions in the schemes  (\textbf{SETD0}) and
(\textbf{SETD1}) are computed either using the real \Leja points technique 
or the Krylov subspace technique. 
For noise with exponential correlations, $b_i>0$, $i=1,2$
we have  $\Vert (-A)^{r/2}Q^{1/2}\Vert_{HS}<\infty$, $r=1,2$. 
Furthermore \assref{assumption3} is obviously satisfied
with $\mathcal{V}=H=L^{2}(\Omega)$ and $\theta =1/2$.
We therefore expect the higher temporal order, i.e. close to $1$ 
with initial data $X_{0}=0$ when
$F$ is taken to be linear.   
We need to consider the projection $P_{h}$ of the noise onto the
computational grid. There are two cases.
When the vertices of our finite element mesh matches the evaluation
points of the noise term $O(t)$ the projection $P_{h}$ is trivial. 
We also used the centered finite volume \cite{FV} discretization. Here 
$P_h$ is trivial when the center of every control volume is an
evaluation point $O(t)$. 
Of course in general the evaluations points of the noise  term $O(t)$
do not necessarily need to match the finite volume or finite element
grids. In this case the noise needs to be regular for a good
projection (see  assumption \ref{assumption3}). 

In our simulations we examined both a finite element and a finite
volume discretization in space and take as a domain $\Omega
=[0,1]\times [0,1]$.  
For time discretizations we compare the schemes here with an
semi-implicit Euler Maruyama method (denoted 'Implicitfem') and 
the semi-implicit Euler Maruyama of \cite{GTambue} that uses linear
functionals of the noise as in \eqref{eq:noiseupdate}.
We denote by 'Implicitfem' the graph for  standard semi-implicit with
finite element method for space discretization with exponential
correlation function,'SETD1fem' and  'SETD0fem' the graph for schemes
(\textbf{SETD1}) and (\textbf{SETD0}) with finite element method for
space discretization with exponential correlation function,
'Implicitfvm$r$', $r=1,2$ the graph for standard implicit with finite
volume method for space discretization with $H^{r}$ noise, SETD1fem$r$
and SETD0fem$r$, $r=1,2$  the graph for the schemes (\textbf{SETD1})
and (\textbf{SETD0}) with finite  element method for space
discretization  with $H^{r}$ noise, SETD1fvm$r$ and SETD0fvm$r$,
$r=1,2$ the graph for the schemes (\textbf{SETD1}) and
(\textbf{SETD0}) with finite volume method for space discretization
with $H^{r}$ noise, ModifiedImplicitfvm$r$, $r=1,2$ graph for the
modified implicit scheme constructed in \cite{GTambue}  
with finite volume method for space discretization with $H^{r}$
noise.


\subsubsection{A linear reaction--diffusion equation}

As a simple example consider the reaction diffusion equation in the
time interval $[0,T]$ with diffusion coefficient $ D>0$  
\begin{eqnarray}
 dX=(D \varDelta X - \lambda X)dt+ dW \qquad X(0)=X_{0},
\label{eq:linear}
\end{eqnarray}
with homogeneous Neumann boundary conditions in $\Omega$. 
Here $\lambda$ is a constant related to the reaction and in the
notation of \eqref{adr} $F(u)=-\lambda u$ and obviously
satisfies condition (a) of Assumption (\ref{assumption4}).
For this linear equation we can construct an exact solution up to any spectral projection error.
We compute the exponential functions $\varphi_{i}$ with the  real fast \Leja points  technique. The absolute tolerance used is $10^{-6}$.

We start by examining in \figref{figII} convergence with $H^{r}$
noise, $r=1,2$. The figure compares the finite element discretization
for schemes (\textbf{SETD0}), (\textbf{SETD1}), the standard implicit
Euler--Maruyama scheme and the modified implicit scheme introduced in
\cite{GTambue} which also uses a linear functional of the noise. 
We observe that schemes with finite element and finite volume space
discretization have the same order of accuracy. 
In \figref{figII} (a) the noise is in $H^{1}$ and the diffusion coefficient is $D=1$. We clearly see improved
accuracy of the schemes that use the linear functions of the noise :
namely (\textbf{SETD0}), (\textbf{SETD1}) and modified implicit over 
the standard semi-implicit method. Not only is there an improved
constant but the temporal order is higher. Numerically we find from
\figref{figII} an order of $0.97 $ for (\textbf{SETD0}),
(\textbf{SETD1}) and  for the modified semi-implicit Euler-Maruyama
scheme, which are in excellent agreement with the theoretical value of
$1$ from the theory, the order of convergence of the standard implicit scheme is $0.30$.
We also see that the scheme
(\textbf{SETD0}) and the modified implicit scheme have approximately
the same order of accuracy and that (\textbf{SETD1}) is slightly more
accurate comparing the schemes (\textbf{SETD0}) and the modified
semi-implicit Euler-Maruyama.
In \figref{figII} (b) the noise is  $H^{2}$ and diffusion coefficient $D=1/100$. 
The error here is dominated by space discretization error, as a consequence to see the convergence 
with need small $\Delta x$ and $\Delta y$.
 We observe again
that the schemes using the linear functionals 
are more accurate. 
 We also see from both \figref{figII} (a) and (b) that
(\textbf{SETD1}) is slightly more accurate than (\textbf{SETD0}) by
some constant. The temporal order of convergence for schemes using linear functional 
of the noise is $0.97$ and $0.5$ for standard semi-implicit scheme.
From \figref{figII} (a) to \figref{figII} (b) we observe that as the
noise is regular the gap between errors in different schemes become small. 
\begin{figure}[!th]
\begin{center}
\includegraphics[width=0.48\textwidth]{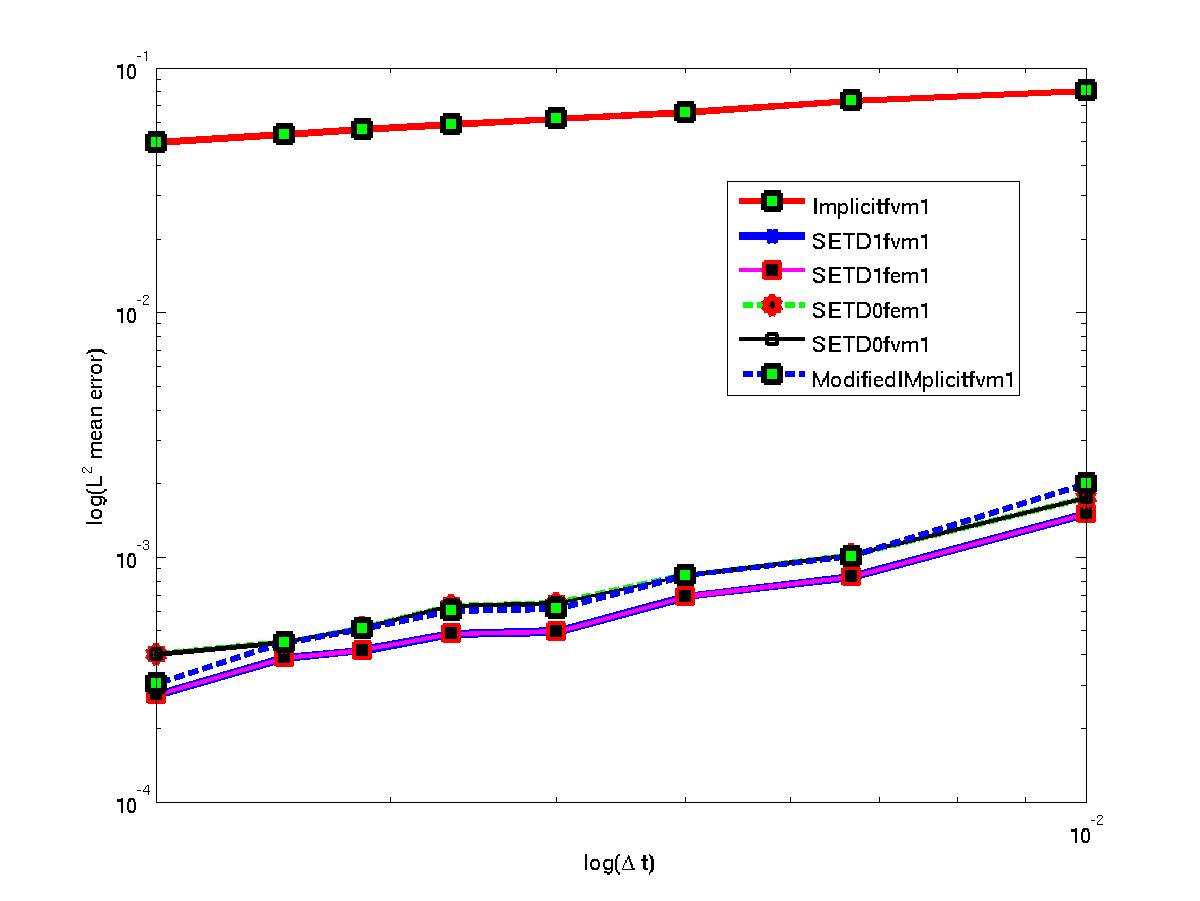}
\includegraphics[width=0.48\textwidth]{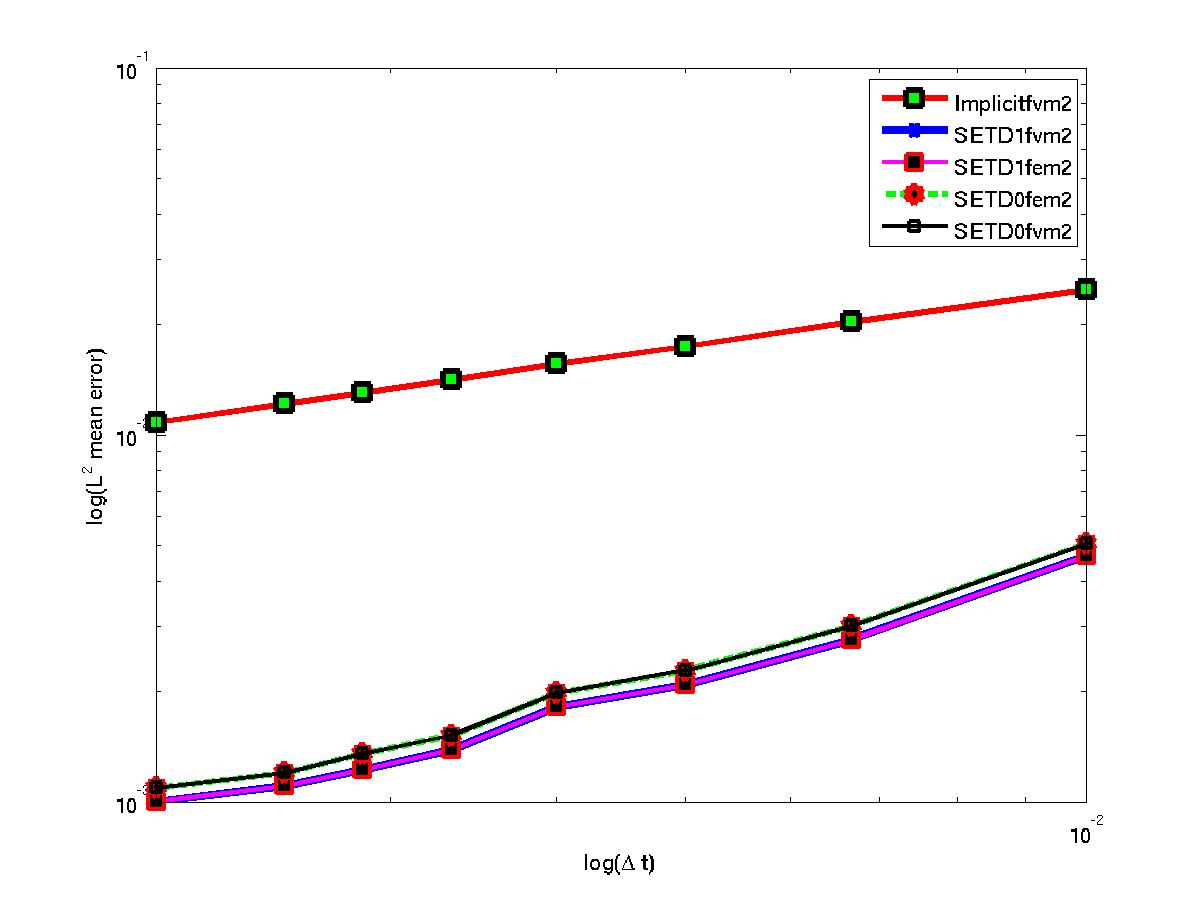}
(a) \hspace{0.48\textwidth} (b)
\end{center}
\caption{Convergence in the root mean square $L^{2}$ norm at $T=1$ as a
  function of $\Dt$ with $H^r$, $r=1,2$. (a) Shows convergence for
  finite element and finite volume discretizations with $r= 1$, $D=1$,
  $\lambda= 1$, $\Gamma =1$ and $\Delta x=\Delta y=1/100$. In (b)
  we show convergence for finite element and finite volume
  discretizations with $r= 2$, $D=1/100$, $\lambda= 1$, $\Gamma =1$,
  $\Delta x =\Delta y =1/400$ (small to have a good look of convergence). 
  The initial data is $X_{0}=0$ and the simulation is for
  \eqref{eq:linear} with 20 realizations.} 
\label{figII} 
\end{figure}

In \figref{figI} we show results with the exponential covariance 
function for the noise, as the noise is certainly in $H^r,\;r= 1
\;\text{or} \;2$ we expect a rate of convergence close to one. The
figure compares the finite element discretization  for schemes
(\textbf{SETD0}) and (\textbf{SETD1})  against the standard implicit
scheme. The temporal order of convergence of  the schemes
(\textbf{SETD0}) is $0.80$ and (\textbf{SETD1}) is $1.05$ and  $0.80$ for  standard implicit scheme.
We see the improved accuracy in the schemes
(\textbf{SETD0}) and (\textbf{SETD1}) comparing to the standard
implicit.  We also see the better accuracy of the scheme (\textbf{SETD1})
compared to (\textbf{SETD0}).

\begin{figure}[!th]
  \begin{center}
    \includegraphics[width=0.48\textwidth]{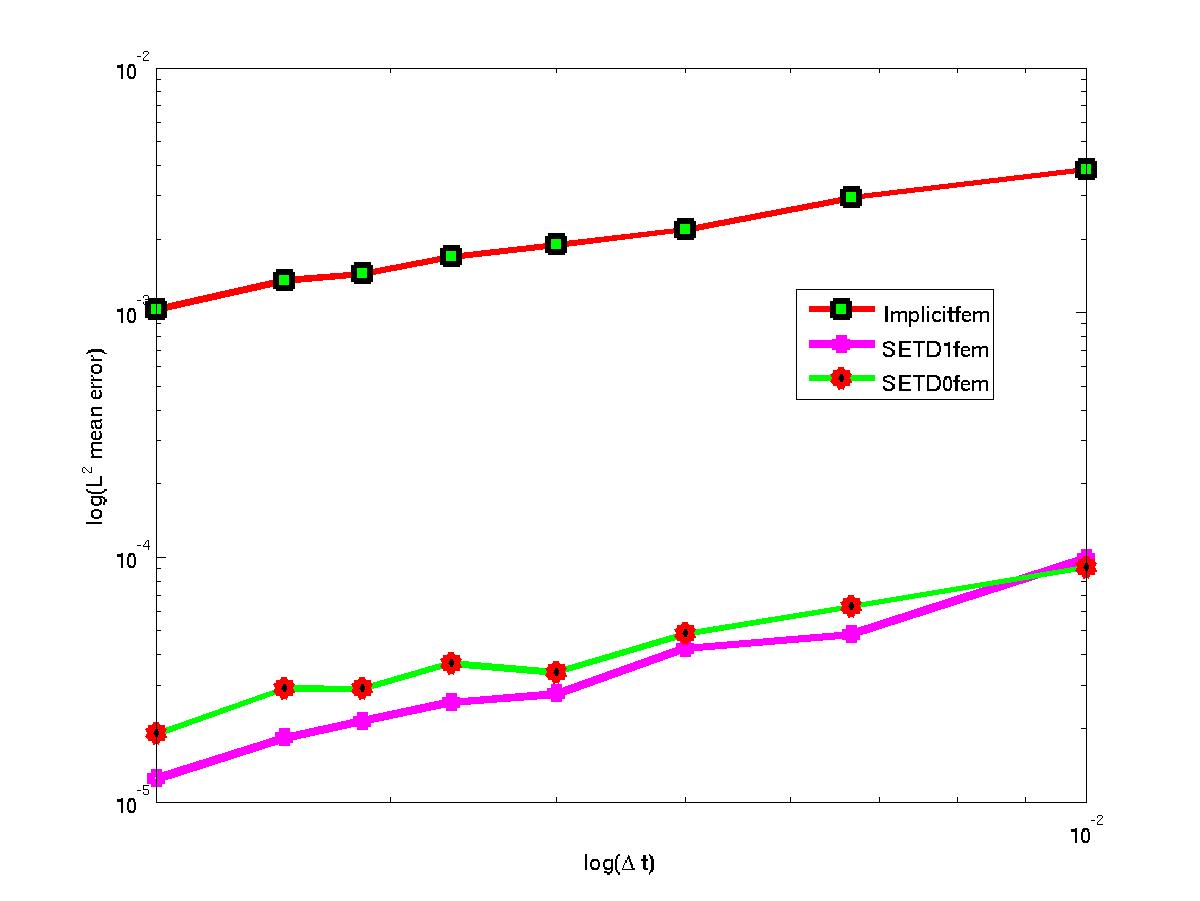}
  \end{center}
  \caption{Convergence in the root mean square $L^{2}$ norm at $T=1$ as a
    function of $\Dt$ with exponential covariance function with $D=1$,
    $\lambda= 0.5$, $\Gamma =1$ and regular mesh coming from  
    rectangular grid with size $\Delta x=\Delta y=1/100$. The simulation
    is for \eqref{eq:linear} with correlation lengths $ b_{1}=b_{2}=0.2$
    and 10 realizations. Initial data is given by $X_{0}=0$.} 
  \label{figI} 
\end{figure}

\subsection{Stochastic advection diffusion reaction}
As a more challenging example we consider the stochastic advection
diffusion reaction SPDE
\begin{eqnarray}
\label{advection}
dX=\left(D \varDelta X -  \nabla \cdot (\textbf{q}
  X)-\frac{X}{X+1}\right)dt+ dW,
\end{eqnarray}
with mixed Neumann-Dirichlet boundary conditions.
and constant velocity 
$\textbf{q}=(1,0)$ for homogeneous medium. 
In terms of equation \eqref{adr} the nonlinear term $F$ is given by
\begin{eqnarray}
 F(u)=-  \nabla \cdot (\textbf{q} u)-\frac{u}{(u+1)}, \quad u \in
 \mathbb{R}^{+} 
\end{eqnarray}
and  clearly satisfies \assref{assumption4} (b).
For heterogeneous medium we used three parallel high 
 permeability streaks. This could represent for example a highly idealized fracture pattern.
  We obtain the Darcy velocity field $\textbf{q}$  by  solving  the system
\begin{eqnarray}
\label{darcy}
 \left\lbrace \begin{array}{l}
 \nabla \cdot\underline{\mathbf{q}} =0\\
\b q=-\dfrac{k(\b x)}{\mu} \nabla p,
\end{array}\right.
\end{eqnarray}
with  Dirichlet boundary conditions 
$\Gamma_{D}^{1}=\left\lbrace 0 ,1 \right\rbrace \times \left[
  0,1\right] \text{and Neumann boundary} 
 \quad \Gamma_{N}^{1}=\left( 0,1\right)\times\left\lbrace 0
   ,1\right\rbrace $ such that 
\begin{eqnarray*}
 p&=&\left\lbrace \begin{array}{l}
1 \quad \text{in}\quad \left\lbrace 0 \right\rbrace \times\left[ 0,1\right]\\
0 \quad \text{in}\quad \left\lbrace L_{1} \right\rbrace \times\left[ 0,1\right]
 \end{array}\right. 
\end{eqnarray*}
and 
$$
- k \,\nabla p (x,t)\,\cdot\b n =0\quad \text{in}\quad \Gamma_{N}^{1}
$$
 where  $p$ is the pressure, $\mu$ is dynamical viscosity and $k$ the
 permeability of the porous medium.  
 We have assumed that rock and fluids are incompressible and sources or
 sinks are absent, thus the equation 
\begin{eqnarray}
  \label{pr}
  \nabla \cdot\underline{\mathbf{q}}=\nabla \cdot 
  \left[\dfrac{k(\b x)}{\mu} \nabla p\right]=0 
\end{eqnarray}
comes from mass conservation. 

To deal with high  P\'{e}clet flows we discretize in space using
finite volumes.  
Simulations are in  $ L^2(\Omega)$ since  the discrete   
$L^2(\Omega)$ norm is easy to implement for all types of boundary conditions. 
We can write the semi-discrete finite volume method as 
\begin{eqnarray}
 dX^{h}=(A_{h}X^{h}+P_{h}F(X^{h})+b(X^{h})) +P_{h}P_N dW,
\end{eqnarray}
where here $A_{h}$ is the space discretization of $D \varDelta $ using
only  homogeneous Neumann boundary conditions and $b(X^{h})$
comes from the approximation of diffusion flux at the Dirichlet
boundary condition size.

We compute the exponential functions $\varphi_{i}$ with Krylov
subspace technique with dimension $m=6$ and the absolute tolerance $10^{-6}$ and the real fast \Leja points  technique for $\varphi_{0}$. 
In  \figref{FIG022a} we
shows the convergence of schemes (\textbf{SETD0}), (\textbf{SETD1})
and standard implicit scheme with $H^{2}$ noise for homogeneous medium, the 'true solution'
is the numerical scheme with smaller time step 
$\Dt=1/15360$. All the schemes have $1/4$ for temporal order of
convergence. We can also observe the  accuracy of the scheme
(\textbf{SETD1}) and (\textbf{SETD0}) comparing to  and standard implicit
scheme in \figref{FIG022a}. 
In  \figref{FIG024a} we
shows the convergence of schemes (\textbf{SETD0}), (\textbf{SETD1})
with $H^{2}$ noise for heterogeneous medium. The two schemes have the
same error. The corresponding mean of CPUtime for the scheme
(\textbf{SETD0}) is given in \figref{FIG024d}. We observe a slightly
efficiency gain using the \Leja points technique  
compared to the Krylov subspace technique during the evaluation of  the action
of $\varphi_{0}$.

In conclusion we obtained superior convergence for the stochastic 
exponential integrators using linear functionals of the
noise with a finite element discretization. Furthermore we have shown
that these schemes that require the exponential of a non-diagonal
matrix can be efficiently implemented for finite element and finite volume
discretizations of realistic porous media flow with stochastic forcing.

\begin{figure}[!th]
  \subfigure[]{
    \label{FIG022a}
    \includegraphics[width=0.48\textwidth]{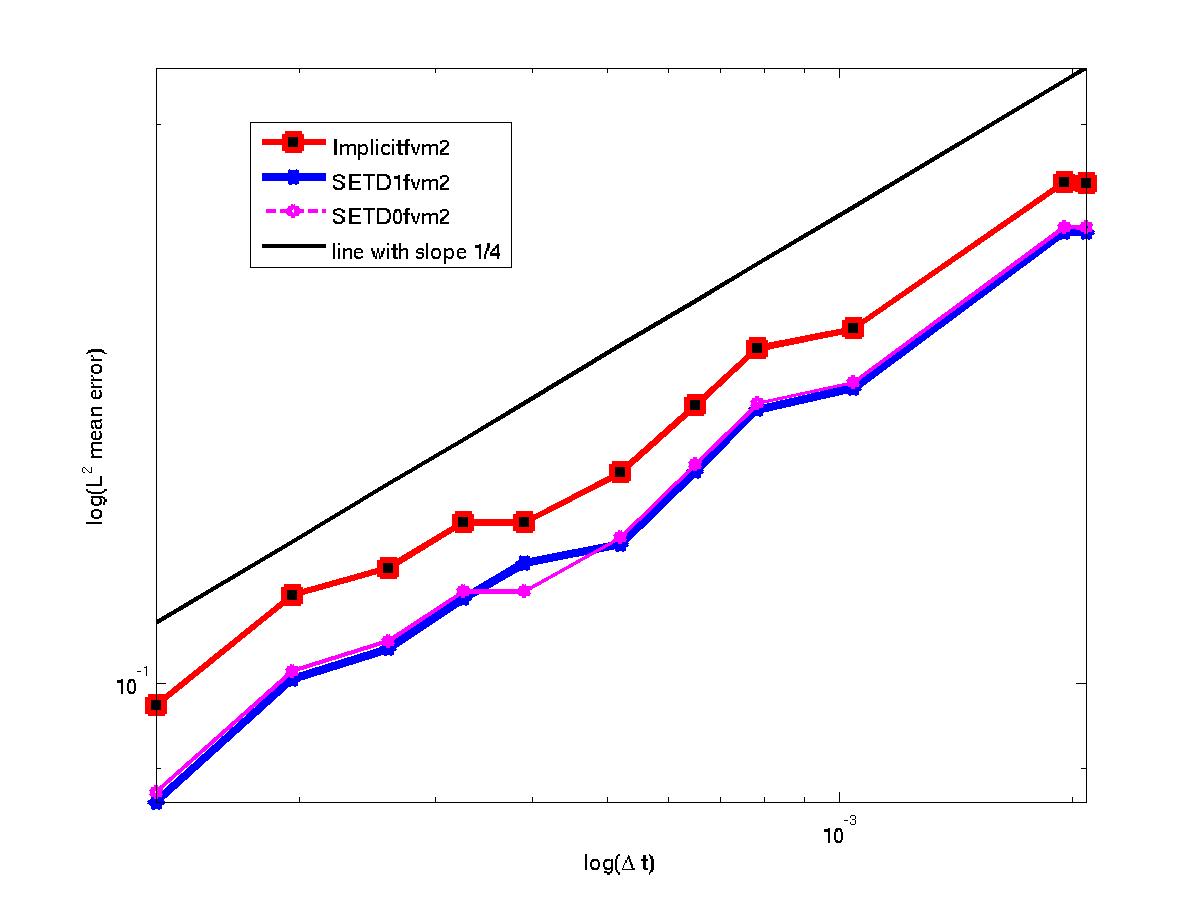}}
  \subfigure[]{
    \label{FIG022b}
    \includegraphics[width=0.48\textwidth]{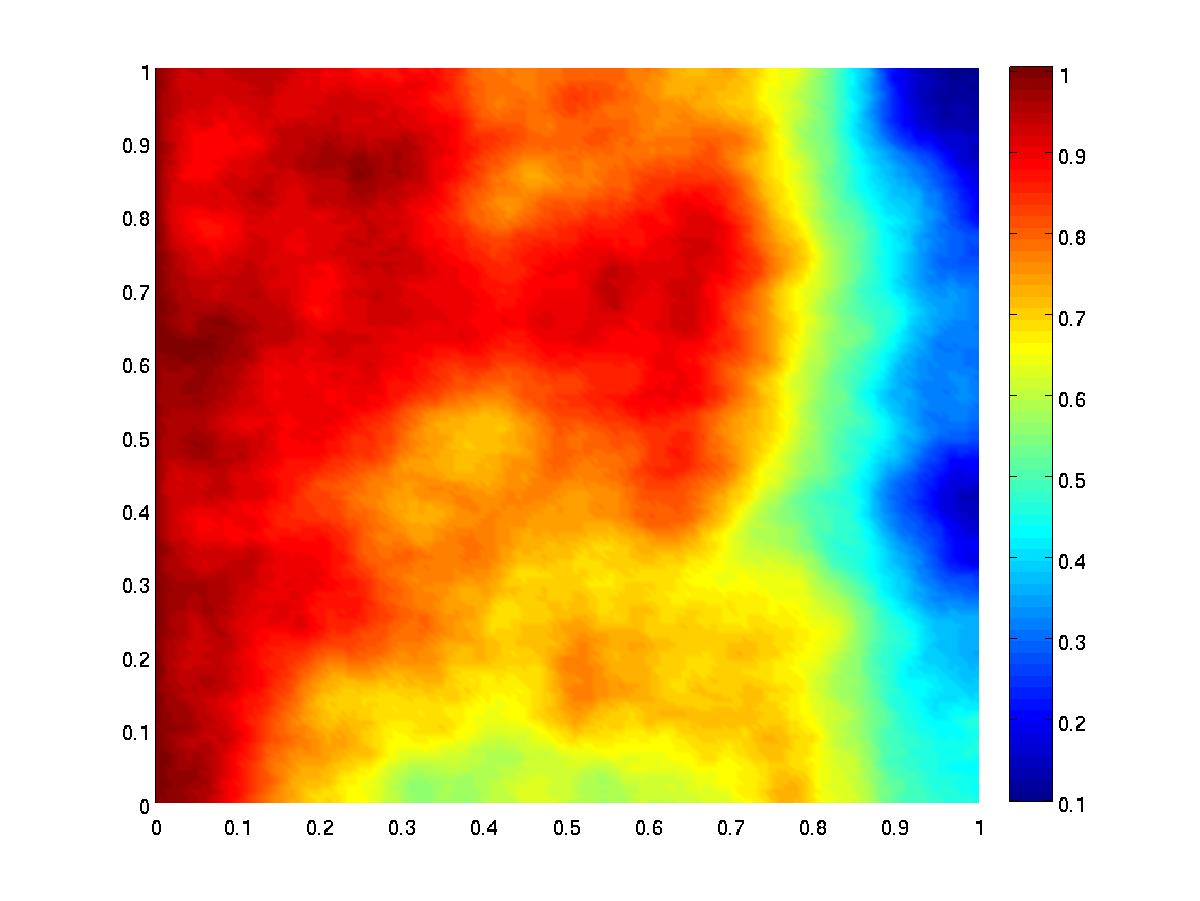}}
  \caption{ (a) Convergence of the root mean square $L^{2}$ norm at $T=1$ as
    a function of $\Dt$ with 30 realizations with $\Delta x= \Delta y
    = 1/160$, $X_{0}=0,\;\Gamma=0.01$ for homogeneous medium. The noise is white in time and in
    $H^{r}$ in space, $r=2$. 
    The temporal order of convergence in time is $1/4$ for all schemes. In (b) we plot a sample 'true solution' for $r=2$ with
    $\Dt=1/15360$.}    
  \label {FIG022}
\end{figure}


\begin{figure}[!th]
  \subfigure[]{
    \label{FIG024a}
    \includegraphics[width=0.48\textwidth]{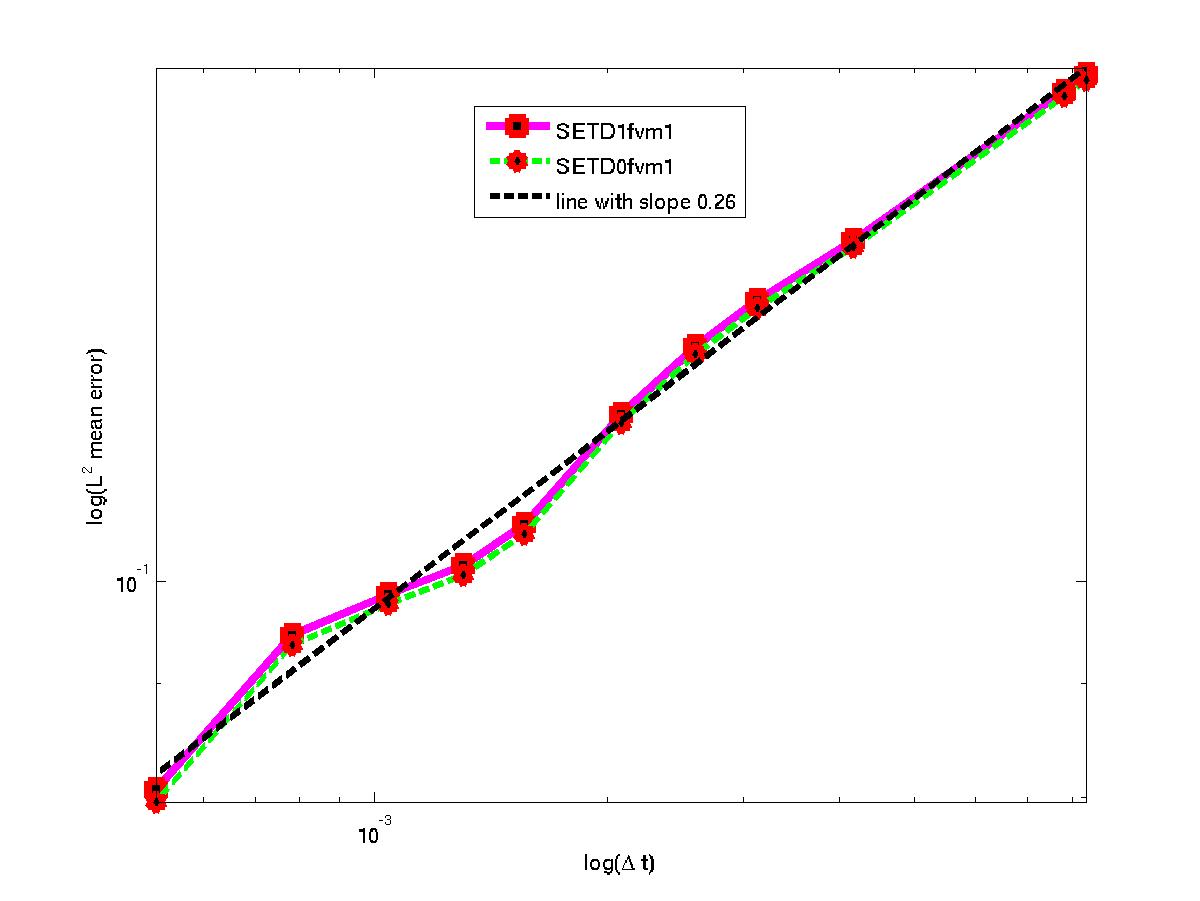}}
  \subfigure[]{
    \label{FIG024b}
    \includegraphics[width=0.48\textwidth]{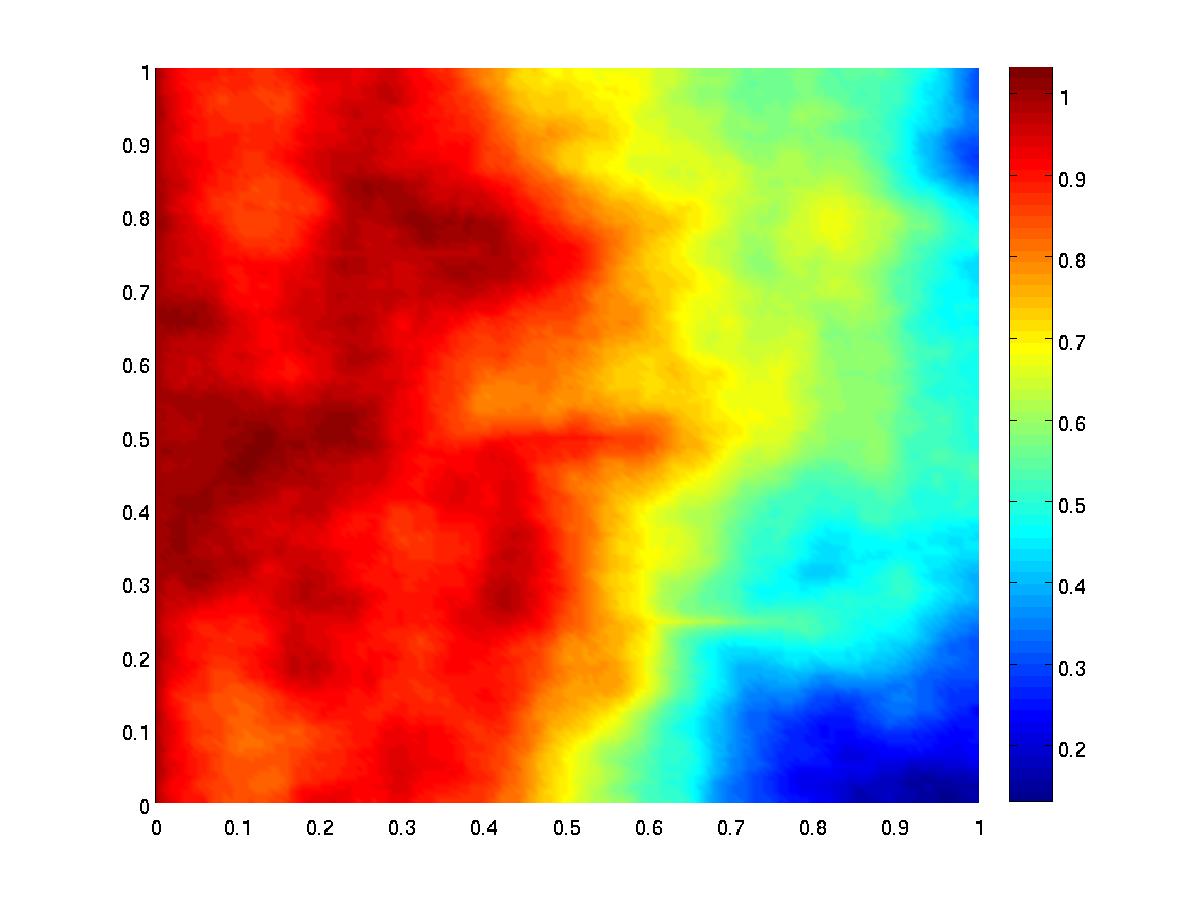}}
\subfigure[]{
    \label{FIG024c}
    \includegraphics[width=0.48\textwidth]{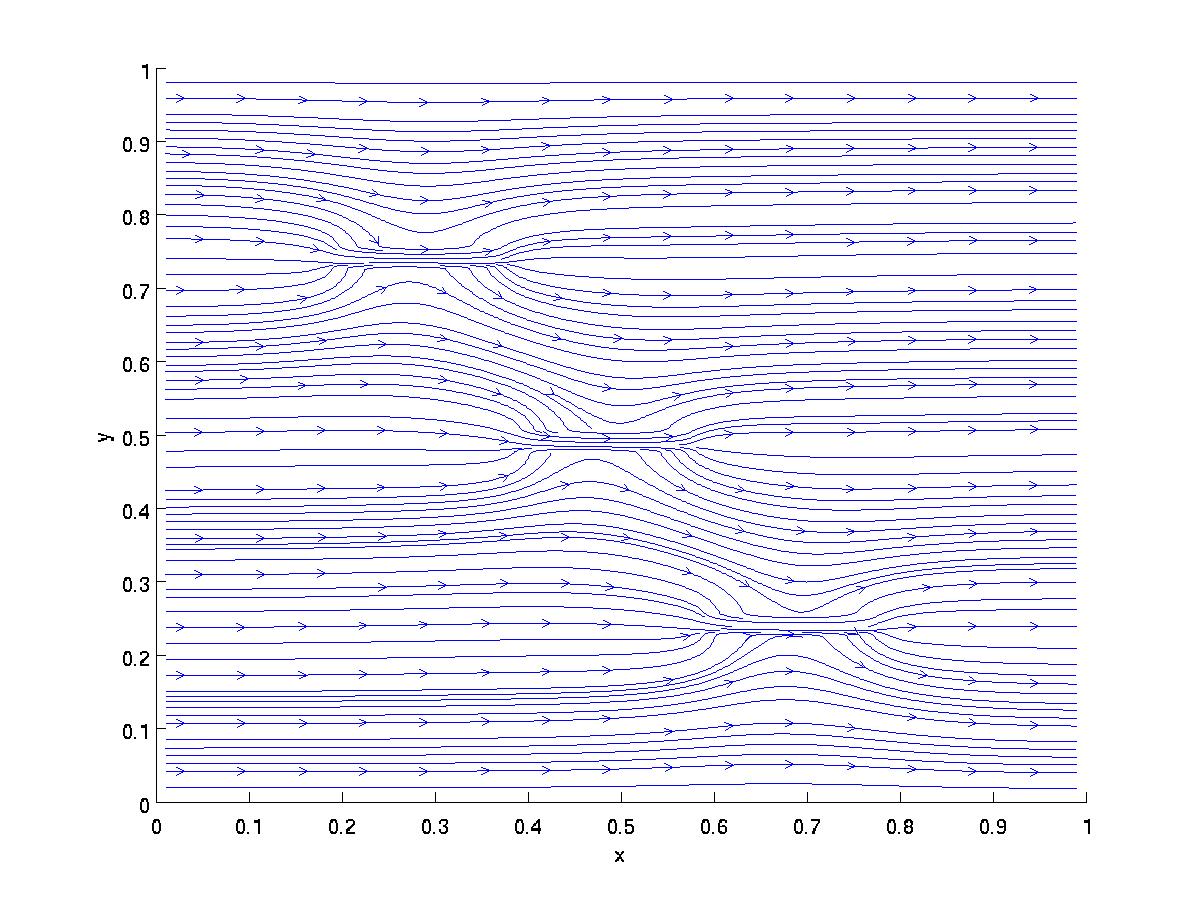}}
\subfigure[]{
  \label{FIG024d}
    \includegraphics[width=0.48\textwidth]{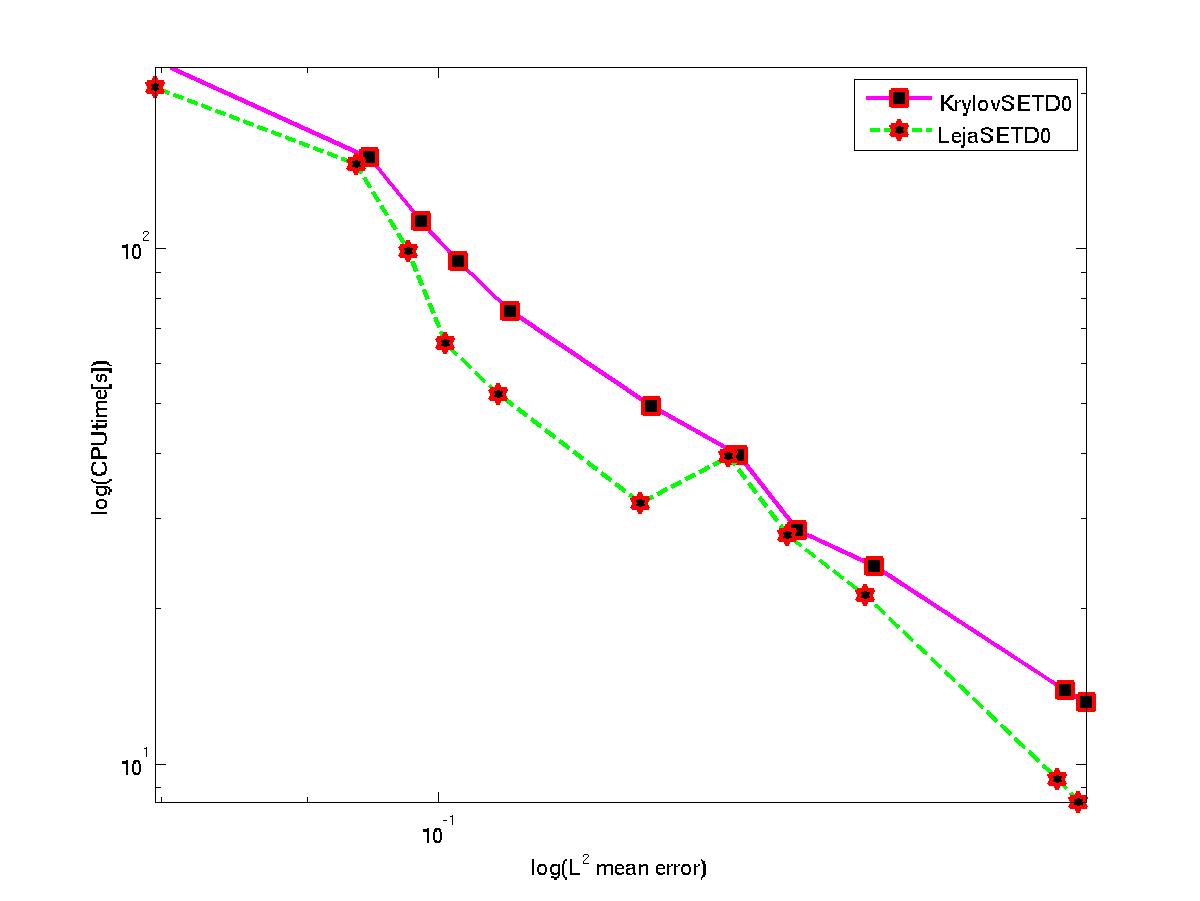}}
  \caption{ (a) Convergence of the root mean square $L^{2}$ norm at $T=1$ as
    a function of $\Dt$ with 30 realizations and $\Delta x= \Delta y
    = 1/160$, $X_{0}=0,\; \Gamma=0.01$ for heterogeneous medium.
    The noise is white in time and in
    $H^{r}$ in space, $r=2$. 
    The temporal order of convergence in time is $0.26$ (close to
    $1/4$) for the two methods. In (b) we plot a sample 'true solution' for $r=1$ with
    $\Dt=1/15360$.  In (c) we plot the streamline of the velocity
    field. In (d) the mean CPUtime for SETD0 using the Krylov and Leja
    points.  The P\'{e}clet number for the flow is $16.58$.}    
  \label {FIG024}
\end{figure}

\bibliographystyle{amsplain}

\begin{thebibliography}{00}

\bibitem{GTambue}
G. J. Lord and A. Tambue.
\newblock  A modified semi--implict Euler-Maruyama scheme for finite
element discretization of SPDEs. 
\newblock {\em Submitted  at  SIAM Journal of Numerical Analysis}, arXiv:1004.1998v1,2010.

\bibitem{Jentzen1}
A.~Jentzen.
\newblock  High order pathwise numerical approximations of SPDES  with additive noise.
\newblock ({\em unpublished manuscript)}, 2009.

\bibitem{Jentzen2}
A.~Jentzen.
\newblock  Pathwise Numerical Approximations of SPDEs
with Additive Noise under Non-global Lipschitz Coefficients.
 \newblock {\em Potential Analysis}, 31(4):375--404,2009.

\bibitem{Jentzen3}
A.~Jentzen, and P.~E. Kloeden.
\newblock Overcoming the order barrier in the numerical approximation of
SPDEs with additive space-time noise.
\newblock {\em Proc. R. Soc. A}, 465(2102)(2009) 649--667.

\bibitem{Jentzen4}
A.~Jentzen, P.~E. Kloeden and G.~Winkel.
\newblock Efficient Simulation of Nonlinear parabolic SPDES with additive noise.
\newblock {\em Submitted to the Annals of Applied Probability }, 2010.

\bibitem{Haus1}
E.~ Hausenblas.
\newblock Approximation for semilinear stochastic evolution equations.
\newblock {\em Potential Analysis}, 18(2)(2003) 141--186.

\bibitem{Jentzen4}
A.~Jentzen, P.~E. Kloeden and G.~Winkel.
\newblock Efficient Simulation of Nonlinear parabolic SPDES with additive noise.
\newblock {\em Submitted to the Annals of Applied Probability }, 2010.

\bibitem{Haus2}
E.~ Hausenblas.
\newblock Numerical analysis of semilinear stochastic evolution equations in Banach spaces.
\newblock {\em J. Comput. Appl. Math}, 147(2) (2002) 485--516 .

\bibitem{Gyongy1}
I.~Gy\"{o}ngy.
\newblock A note on Euler's approximations.  
\newblock {\em Potential Anal. }, 8(3)(1998) 205-216 .

\bibitem{Gyongy2}
I.~Gy\"{o}ngy, and A.~Millet.
\newblock  On discretization schemes for stochastic evolution equations.
\newblock {\em Potential Anal. },23(2)(2005) 99-134 .

\bibitem{FV}
R.~Eymard, T.~Gallouet, and R.~Herbin,
\newblock Finite volume methods,
\newblock in: P.~G.~Ciarlet, J.~L.~Lions (Eds.),
\newblock Handbook of Numerical Analysis Volume 7, North-Holland, Amsterdam, 2000, pp. 713--1020.

 
\bibitem{EP}
P.~Knabner and L.~Angermann.
\newblock Numerical methods for elliptic and parabolic partial differential equations solution.
\newblock {\em Springer}, 2003.
 
\bibitem{Henry}
D.~Henry.
\newblock Lecture Note in Mathematics: Geometric Theory of Semilinear Parabolic Equations. 
\newblock {\em Springer}, 840, 1981.

\bibitem{Stig}
S.~ Larsson.
\newblock Nonsmooth data error estimates with applications to
the study of the long-time behavior of finite element
solutions of semilinear parabolic problems. 
\newblock {\em Preprint 1992-36, Department of Mathematics, Chalmers University of
Technology, available at http://www.math.chalmers.se/$\sim$stig/papers/index.html }.
 


\bibitem{Vidar}
V.~ Thom\'{e}e.
\newblock Galerkin finite element methods for parabolic problems. 
\newblock {\em Springer Series in Computational Mathematics}, 1997.

\bibitem{lions}
H.~ Fujita, and T.~ Suzuki.
\newblock  Evolutions problems (part 1) in: P.~G.~Ciarlet, J.~L.~Lions (Eds.). 
\newblock {\em Handbook of  Numerical Analysis }, vol II, North-Holland, Amsterdam, pp. 789--928, 1991 .

\bibitem{Antoine}
A. Tambue, G. J. Lord, and S. Geiger.
\newblock  An exponential integrator for advection-dominated reactive
transport in heterogeneous porous media.  
\newblock {\em Journal of Computational Physics },229(10)(2010)  3957-- 3969.

\bibitem{ATthesis}
A. Tambue
\newblock Efficient Numerical Methods for Porous Media Flow.
\newblock {\em Department of Mathematics, Heriot--Watt University}, 2010.

\bibitem{KLNS}
P. Kloeden and G. J. Lord and A. Neuenkirch and T. Shardlow.
\newblock The exponential integrator scheme  for  stochastic partial
differential equations:  Pathwise error bounds. 
\newblock  {\em Submitted to J. Comp. A. Math.}, 2009.

\bibitem{LT}
 G. J. Lord, and  T.Shardlow.
 \newblock Postprocessing for stochastic parabolic partial
 differential equations. 
\newblock {\em SIAM J. Numer. Anal.},45(2) (2007)  870--889.

\bibitem{LR}
G. J. Lord and J. Rougemont.
\newblock A numerical scheme for stochastic PDEs with Gevrey regularity. 
\newblock  {\em IMA J. Num. Anal.},24(4)(2004) 587--604 .

\bibitem{sebastianb}
P.B. Bedient and H.S. Rifai and C.J. Newell.
\newblock Ground  Water Contamination: Transport and Remediation.
\newblock  {\em Prentice Hall PTR , Englewood Cliffs}, New Jersey 07632, 1994.

\bibitem{PrvtRcknr}
C.Prevot and M. Rockner.
\newblock A Concise Course on Stochastic Partial Differential Equations.
 \newblock  {\em Springer}, 2007.

\bibitem{DaPZ}
G. Da Prato,  and J. Zabczyk.
\newblock Stochastic Equations in Infinite Dimensions.
\newblock {\em Encyclopedia of Mathematics and its Applications, 45 Cambridge University Press}, 1992.

\bibitem{Pazy}
A. Pazy.
\newblock Semigroups of linear operators and applications to partial  differential equations.
\newblock {\em Applied Mathematical Sciences, Springer-Verlag, New York}, 1983.

 \bibitem{ElliottLarsson}
C. M. Elliott and S. Larsson.
\newblock Error estimates with smooth and nonsmooth data for a finite element method for the Cahn-Hilliard equation.
\newblock {\em Math. Comp. }, 58 (1992)  603--630.

\bibitem{Yn:04}
Y. Yan.
\newblock Semidiscrete Galerkin approximation for a linear stochastic
parabolic partial differential equation driven by an additive noise. 
\newblock {\em BIT}, 44(4) (2004)  829--847.

\bibitem{shardlow05}
T. Shardlow.
\newblock Numerical simulation of stochastic {PDE}s for excitable media.
\newblock {\em J. Comput. Appl. Math},175(2)(2005) 429--446.

\bibitem{GrcaOjlvoSncho}
J. Garc{\'{\i}}a-Ojalvo and  J.M. Sancho.
\newblock Noise in spatially extended systems.
\newblock {\em Institute for Nonlinear Science, Springer-Verlag, New York}, 1999.

\bibitem{WuanT}
W. Luo.
\newblock Wiener chaos expansion and numerical solutions of stochastic
partial differential equations. 
\newblock {\em California Institute of Technology ,Pasadena, California }, 2006.

\bibitem{LE1}
J.~Baglama,~D.~Calvetti, and L.~Reichel,
\newblock Fast \Leja points,
\newblock {\em Electron. Trans. Num. Anal.} 7 (1998) 124--140.

\bibitem{LE}
L.~Bergamaschi,~M.~Caliari, and M.~Vianello,
\newblock The {RELPM} exponential integrator for {FE} discretizations of
  advection-diffusion equations,
\newblock in: M.~Bubak, G.~D.~ Van Albada, P.~Sloot (Eds.),
\newblock Lecture Notes in Computer Sciences Volume 3039, Springer Verlag, Berlin Heidelberg, 2004, pp. 434-442.

\bibitem{kry}
M.~Hochbruck and C.~Lubich,
\newblock On {K}rylov subspace approximations to the matrix exponential
  operator,
\newblock {\em SIAM J. Numer. Anal.} 34(5) (1997) 1911--1925.

\bibitem{SID}
R.~B. Sidje,
\newblock Expokit: {A} software package for computing matrix exponentials,
\newblock {\em ACM Trans. Math. Software} 24(1) (1998) 130--156.

\bibitem{CMCVL}
C. Moler and C.~Van Loan,
\newblock Ninteen Dubious Ways to Compute the Exponential of a
Matrix, Twenty--Five Years Later,
\newblock {\em SIAM Review} 45(1) (2003) 3--49.

\bibitem{LEJAK}
A.~Martinez, ~L.~Bergamaschi, M.~Caliari, and M.~Vianello,
\newblock A massively parallel exponential integrator for advection-diffusion
  models,
\newblock {\em J. Comput. Appl. Math.} 231(1) (2009) 82--91.

\bibitem{LEJAK1}
L.~Bergamaschi, M.~Caliari,~A. Martinez, and M.~Vianello,
\newblock Comparing \Leja and {K}rylov approximations of large scale
  matrix exponentials,
\newblock {\em Comput. Sci. -- ICCS} 3994 (2006) 685--692.

\bibitem{HSW}
H.~Berland, B.~Skaflestad, and W.~Wright,
\newblock A matlab package for exponential integrators,
\newblock {\em ACM Trans. Math. Software} 33(1) (2007) Article No. 4

\bibitem{LE2}
M.~Caliari, M.~Vianello, and L.~Bergamaschi,
\newblock Interpolating discrete advection diffusion propagators at \Leja
  sequences,
\newblock {\em J. Comput. Appl. Math.} 172(1) (2004) 79--99.

\bibitem{GC}
J.~W.~Thomas,
\newblock Numerical partial differential equations: finite difference methods,
\newblock Springer Verlag, Berlin Heidelberg  New York, 1995.

\bibitem{Reichel}
L.~Reichel.
\newblock Newton interpolation at Leja points. 
\newblock {\em BIT }, 30(2)(1990)  332--346.

\bibitem{McM}
A.~McCurdy, K.~C.~Ng, and B.~N.~Parlett.
\newblock Accurate computation of divided differences of the exponential function.
\newblock {\em Math. Comp.}, 43(186)(1984) 501--528.
\end{thebibliography}

\end{document}